\documentclass[a4paper,11pt]{amsart}

\usepackage[top=25truemm,bottom=25truemm,left=35truemm,right=35truemm]{geometry}
\usepackage{mathtools}

\DeclarePairedDelimiter{\red}{(}{)_{\operatorname{red}}}
\usepackage{amsmath}
\usepackage{amssymb}
\usepackage{amsthm}
\usepackage{amscd}
\usepackage{color}
\usepackage{ulem}
\usepackage{hyperref}
\usepackage{comment} 

\theoremstyle{plain}
\newtheorem{thm}{Theorem}[section]

\newtheorem{lem}[thm]{Lemma}
\newtheorem{cor}[thm]{Corollary}

\theoremstyle{definition} 

\newtheorem{defi}[thm]{Definition}
\newtheorem{defnlemma}[thm]{Definition--Lemma}

\newtheorem{rmk}[thm]{Remark}
\newtheorem{ex}[thm]{Example}

\begin{document}

\subjclass[2020]{Primary 14E30; Secondary 14B05}

\keywords{Generalized Pair, Minimal Log Discrepancy, ACC Conjecture, Boundedness}

\title[Generalized mld]{On the ACC for minimal log discrepancies for bounded generalized Sub-Pairs}
\author{Weichung Chen}
\author{Keng-Hung Steven Lin}

\address{Department of Mathematics, National Taiwan University,
Astronomy Mathematics Building 5F, No. 1, Sec. 4, Roosevelt Rd., Taipei 10617, Taiwan}
\email{weichungchen@ntu.edu.tw}


\address{Graduate Institute of Communication Engineering, National Taiwan University, BL-501, No. 1, Sec. 4, Roosevelt Road, Taipei 10617, Taiwan}
\email{khslin@ntu.edu.tw}

\begin{abstract}
We show that the ACC conjecture for minimal log discrepancies holds for both bounded generalized sub-pairs in arbitrary dimension and bounded threefolds with arbitrary boundaries.
\end{abstract}

\maketitle

\tableofcontents

\section{Introduction}
The \emph{ascending chain condition (ACC) for minimal log discrepancies} is well-known as being a crucial topic for solving the \emph{termination of flips} conjecture in the minimal model program (MMP). 

Shokurov proves the following theorem.
\begin{thm}[\cite{ShoVII09}]
If the following conditions hold.
\begin{itemize}
    \item{(ACC for Minimal Log Discrepancies, \emph{ACC for mld})} Let $d\in\mathbb{Z}_{>0}$ and $I\subseteq [0,1]$ be a set satisfying the descending chain condition (DCC). Then the set 
    \[
    \{\operatorname{mld}(X,B)|\dim(X)=d, (X,B)\text{ is a pair such that }\operatorname{coeff}(B)\in I\}
    \]satisfies the ascending chain condition (ACC),
    \item{(Lower Semi-Continuity for Minimal Log Discrepancies, \emph{LSC for mld})} For a pair $(X,B)$, the function 
    \[
    \operatorname{mld}_{{\bullet}}(X,B):X\to\mathbb{R}\cup\{\infty\}
    \]is lower semi-continuous on closed points $x\in X$. 
\end{itemize}
Then the flips terminate.
\end{thm}
The ACC for mld is known to be held for the following scenarios.
\begin{itemize} 
    \item Surfaces \cite{VA93}.
    \item Toric pairs \cite{Amb06}.
    \item Varieties with a fixed Gorenstein index \cite{Nak16}.
    \item Smooth threefolds \cite{kaw26}. 
    \item Fixed threefolds \cite{kaw26fix3fold}.
    \item Generalized pairs for surfaces \cite{CGN24}.
\end{itemize}
In this paper, we prove that the ACC for both mld holds for bounded threefolds and bounded generalized sub-pairs in arbitrary dimension.

\begin{thm}[ACC for mld for Bounded Generalized Sub-Pairs]
Fix natural numbers $d$ and $n$ and fix a DCC set $I\subseteq \mathbb{R}_{\geq 0}$. Let $P(d,n,I)$ be the set of generalized sub-pairs $(X,B+\mathbf{M})$ such that 
\begin{itemize}
\item $\dim X = d$,
\item $X$ is $\mathbb{Q}$-factorial,
\item we may write $B=E-F$ for some some effective divisors $E$ and $F$,
\item there is a very ample divisor $H$ on $X$ with $H^d\leq n$ and $H-(E+F+\mathbf{M})\geq 0$,
\item $F$ is integral, and
\item $\mathrm{coeff}(E)\in I$ and $\mathrm{coeff}(\mathbf{M})\in I$.
\end{itemize}
Then the set
\[\{ \operatorname{mld}(X,B+M)|(X,B+M)\in P(d,n,I)\}\]
satisfies ACC.  
\end{thm}

\begin{thm}[ACC for mld for Bounded Pairs of Threefolds]
Fix a DCC set $I\subseteq\mathbb{R}_{>0}$ and $N\in\mathbb{Z}_{>0}$, then there is an ACC set $J=J(I,N)\subseteq\mathbb{R}_{\geq 0}$ satisfying the following statement: If $(X,B)$ is a lc pair of dimension $3$, such that
\begin{itemize}
    \item $\operatorname{coeff}(B)\in I$, and
    \item there is a very ample divisor $H$ on $X$ with $H^3\leq N$,
\end{itemize}
then $\operatorname{mld}(X,B)\in J$.
\end{thm}

\section{Preliminaries}\label{section:pre}
We recall some necessary definitions and notion. We always work over $\mathbb{C}$, and varieties are always considered to be normal projective $\mathbb{Q}$-factorial integral and separated schemes over $\mathbb{C}$, unless stated otherwise. Such assumptions have no cost in the problems involved in minimal log discrepancies because of replacements by small $\mathbb{Q}$-factorialization preserves the minimal log discrepancies.

There are some notation and abbreviations which will be used in the rest of the article broadly.
\begin{itemize}
    \item ACC = ascending chain condition
    \item DCC = descending chain condition
    \item $\mathrm{mld}$ = minimal log discrepancy
    \item For a divisor $B$, we express $B=\sum_{i=0}^{m}b_{i}B_{i}$, where $B_{i}$'s are prime divisors. Let $I\subseteq\mathbb{R}$, then the notation
    \[
    \operatorname{coeff}(B)\in I
    \]means that $b_{i}\in I$, for all $i=0,1,\cdots,m$.
\end{itemize}

A proper morphism $f:X\to Y$ is a \emph{contraction} if $f_{*}\mathcal{O}_{X}=\mathcal{O}_{Y}$.

A family of normal projective varieties $\mathcal{P}$ is \emph{bounded} if there exist finitely many projective morphisms $\{V^{i}\to T^{i}\}$ of varieties such that for each $X\in\mathcal{P}$, there is a $j$, a closed point $t\in T^{j}$, and an isomorphism $\phi:V^{i}_{t}\dashrightarrow X$, where $V^{j}_{t}$ is the fiber of $V^{j}\to T^{j}$ over $t$.

\begin{defi}(Image of Divisors)\label{imofd}
Let $\psi:X\dashrightarrow Y$ be a birational map between varieties.
Let $W$ be a variety with birational morphisms $f:W\to X$ and $g:W\to Y$ that resolves $\psi$.
For an $\mathbb{R}$-Cartier divisor $M$ on $X$,  the \emph{image of $M$ on $Y$ through $\psi$} is the $\mathbb{R}$ divisor $g_*f^*M$ on $Y$.
It's clear that this definition depends on $\psi$ but not on the choice of the resolution $W$ of $\psi$. We omit $\psi$ and denote the image of $M$ on $Y$ by \emph{$M|_Y$} if there is no ambiguity with the birational map $\psi$.
\end{defi}
\begin{ex}
Let $f:Y\to X$ be a morphism, and let $M$ and $N$ be $\mathbb{R}$-Cartier divisors on $X$ and $Y$, respectively. Then we have $M|_{Y}=f^{*}M$, which is the usual pullback of $M$ to $Y$. $N|_{X}=f_{*}N$, which is the usual pushforward of $N$.   
\end{ex}

\begin{defi}
 Let $X$ be a normal variety. A \emph{b-divisor} $\mathbf{M}$ over $X$ is an equivalent class of divisors on birational models over $X$ under the following relation. Divisors $M_1$ on $X_1\to X$ and $M_2$ on $X_2\to X$ are in the same class if and only if there is a common resolution $Y\to X_i$, where $i=1,2$, such that the pullbacks of $X_i$ on $Y$ coincide. 
 On any birational model $X'$ over $X$, the \emph{image} $\mathbf{M}|_{X'}$ of $\mathbf{M}$ on $X'$ is the image $M_1|_{X'}$ of $M_1$ on $X'$, which is independent of the choice of $M_1\subseteq X_1$.

 For any $\mathbb{R}$-Cartier divisor $E$ over $X$, the b-divisor $[E]$ over $X$ is defined as the equivalent class containing $E$.

 On the other hand, given a b-divisor $\mathbf{M}$ over $X$, we say that $\mathbf{M}$ descends to some divisor $L\subseteq Y$ on some birational model $Y$ over $X$ if $\mathbf{M}=[L]$.
\end{defi}

\begin{defi}[Sub-Pairs/Pairs]
A \emph{sub-pair} $(X, B)$ consists of
\begin{itemize}
    \item a normal quasi-projective variety $X$, and
    \item an $\mathbb{R}$-divisor $B$ whose coefficients in $(-\infty,1]$ and $K_{X} + B$ is $\mathbb{R}$-Cartier. 
\end{itemize}
 A sub-pair $(X, B)$ is called a \emph{pair} if furthermore $B$ is effective. That is, the coefficients of $B$ are in $[0,1]$.
\end{defi}
\begin{defi}(b-nef over $Z$ and b-$\mathbb{R}$-Cartier)
    Let $X\to Z$ be a morphism of normal varieties.
The b-divisor $\mathbf{M}$ over $X$ is called \emph{b-nef over $Z$} (\emph{b-$\mathbb{R}$-Cartier}, respectively) if it descends to some divisor $L$ on some birational model $Y$ over $X$ such that $L$ is nef over $Z$ (\emph{$\mathbb{R}$-Cartier}, respectively). In the rest of the article, the term ``b-nef over $Z$" will be abbreviated as \emph{b-nef$/Z$}.
\end{defi}
\begin{defi}(Generalized Sub-Pairs/Generalized Pairs)
A \emph{generalized sub-pair} $(X, B + \mathbf{M})/Z$ consists of
\begin{itemize}
    \item a normal variety $X$ and a quasi-projective variety $Z$ with a projective morphism $X \to Z$,
    \item  an $\mathbb{R}$-divisor $B$ on $X$, and
    \item a b-$\mathbb{R}$-Cartier b-divisor $\mathbf{M}$ and over $X$ descending to a nef$/Z$ $\mathbb{R}$-Cartier divisor $M$ on some birational model $X'\xrightarrow{\varphi} X$,
\end{itemize}
such that $K_{X}+B+M$ is $\mathbb{R}$-Cartier, where $M=\mathbf{M}|_X=\varphi_{*}M'$.
We usually omit $Z$ if $Z$ is a closed point.

A generalized sub-pair $(X, B + \mathbf{M})/Z$ is a \emph{generalized pair} if furthermore $B$ is effective.
\end{defi}

We recall the definition of log discrepancies for a generalized sub-pair. 

\begin{defi}
Let $(X,B+\mathbf{M})/Z$ be a generalized sub-pair. For a prime divisor $E$ over $X$, we define the \emph{generalized log discrepancy} $a_E (X/Z, B + \mathbf{M})$ of $(X,B+\mathbf{M})/Z$ with respect to $E$ as follows. 
Take a biraional model $\varphi :X' \to X$ over $X$ such that $\varphi$ is a log resolution of $(X, B)$, $E$ is a divisor on $X'$, and $\mathbf{M}$ descends to some divisor $M'$ on $X'$. 
Let $B'$ be the $\mathbb{R}$-divisor on $X'$ such that 
\[
K_{X'} + B' + M' = \varphi ^* (K_X + B +\varphi_*M'). 
\]
Then we define $a_E (X/Z, B + \mathbf{M}) := 1 - \operatorname{coeff}_E B'$.
Since $Z$ is not involved in the computations of $a_E (X/Z, B + \mathbf{M})$, we usually omit it in the notations of log discrepancies unless we want to emphasize that $\mathbf{M}$ is b-nef$/Z$.

We say that a generalized sub-pair $(X,B + \mathbf{M})$ is \emph{sub log canonical} (sub-lc) if $a_E (X, B + \mathbf{M})\ge 0$ holds for every prime divisor $E$ over $X$.

 Let $\eta\in X$ be a point. The \emph{minimal log discrepancy} $\operatorname{mld}_\eta(X,B+\mathbf{M})$ at of a generalized sub-lc pair $(X,B+\mathbf{M})/Z$ at $\eta$ is defined to be $\min\{a_E(X,B+\mathbf{M})\}_E$, where $E$ runs over all prime divisors over $X$ such that $c_X(E)\subseteq\overline{\{\eta\}}$. When $\overline{\{\eta\}}=X$, we usually omit the index $\eta$ and write $\operatorname{mld}(X,B+\mathbf{M})$.
\end{defi}

\begin{defi}
Let $(X,B+\mathbf{M})/Z$ be a generalized set.
For a set of real numbers $I$, we say that $\emph{\operatorname{coeff}(\mathbf{M})\in I}$ if $\mathbf{M}$ descends to some divisor $M'=\sum r_iM'_i$ on a birational model $Y$ over $X$ for some nef$/Z$ Cartier divisors $M'_i$ on $Y$ and some non-negative real numbers $r_i \in I_{\geq0}$.
We say that $\emph{\operatorname{coeff}(B)\in I}$ if all non-zero coefficients (of prime components) of $B$ belong to $I$.
\end{defi}

The following theorem shows the discreteness for lc pairs with a fix Gorenstein index.
\begin{thm}[\cite{CGN24}, Theorem 6.5]\label{discrete}
Let $d$ and $r$ be natural numbers and $I_0$ be a finite set of positive real numbers. Then there is a closed discrete set $J\subseteq \mathbb{R}_{\geq 0}$ satisfying the following.

If $(X,B+\mathbf{M})$ is a generalized lc pair of dimension $d$ such that
\begin{itemize}
    \item $rK_X$ is Cartier,
    \item $\operatorname{coeff}(B)\in I_0$,
    \item $\mathbf{M}$ is a b-nef b-$\mathbb{R}$-Cartier divisor over $X$ with $\operatorname{coeff}(\mathbf{M})\in I_0$, and
    \item $E$ is a prime divisor over $X$,
\end{itemize}
then $a_E(X,B)\in J$. In particular, we have $\operatorname{mld}(X,B)\in J$.
\end{thm}

\begin{defi}[Ascending/Descending Chain Condition]
Let $J\subseteq\mathbb{R}$ be a set of real numbers. We say that $J$ satisfies
\begin{itemize}
    \item the \emph{ascending chain condition} (ACC) if every non-empty subset of $J$ contains a maximal element.
    \item the \emph{descending chain condition} (DCC) if every non-empty subset of $J$ contains a minimal element.
\end{itemize}
\end{defi}

The following lemma follows from projection formula in intersection theory.
\begin{lem}\label{basic}
    Let $X$ be a smooth threefold, $S$ be a smooth surface on $X$, $C\subseteq S$ be a smooth curve, and $x\in C$ be a closed point.

    \begin{enumerate}
        \item Let $X_C$ be the blow-up of $X$ along $C$ with exceptional divisor $E_C$ and $l=E_C\cap\tilde{S}$. Then $E_C\cdot_{X_C} l=C\cdot_S C$.

    \item Let $X_x$ be the blow-up of $X$ at $x$. Then $\tilde{C}\cdot_{\tilde{S}}\tilde{C}=C\cdot_S C-1$.

    \end{enumerate}
\end{lem}

\begin{defi}
    Let $X$ be a variety and let $D$ be a divisor on $X$.
    Write $D=\sum_i{d_i}D_i$ for some non-zero real numbers $d_i$ and prime divisors $D_i$ on $X$.
    Then the \emph{reduced divisor associated to $D$}, denoted $\red{D}$, is defined to be $\red{D}=\bigcup_i D_i$.
\end{defi}

\begin{defi}(Center of Divisors and b-Divisors over a Variety)
Let $X$ be a variety and let $M'$ be an $\mathbb{R}$-divisor on a birational model $X'$ over $X$.
Write $M'=\sum_{i\in I} m_iM'_i$ for distinct prime divisors $M'_i$ on $X'$, a finite set $I$ and non-zero real numbers $m_i$.
The \emph{center $c_X(M')$} of $M'$ on $X$ is the union $\bigcup_{i\in I} c_X(M'_i)$ of subvarieties of $X$.

For a b-$\mathbb{R}$-Cartier b-divisor $\mathbf{M}$ over $X$, we define the \emph{center $c_X(\mathbf{M})$} of $\mathbf{M}$ to be $c_X(M_{X'})$, where $M_{X'}$ is a divisor that $\mathbf{M}$ descends to on a birational model $X'$ over $X$. It is clear that this definition does not depend on the choice of $X'$. 
\end{defi}

\begin{defi}(Order of a b-divisor along a prime divisor over $X$)
Let $X$ be a variety and let $\mathbf{M}$ be a b-divisor over $X$ such that its image $M=\mathbf{M}|_X$ on $X$ is a $\mathbb{R}$-Cartier divisor.
Let $E$ be a prime divisor over $X$.
Then we define the \emph{order} $\operatorname{ord}_{E\slash X} (\mathbf{M})$ of $\mathbf{M}$ at $E$ as follow. Let $\varphi:X'\to X$ be a birational model over $X$ such that $\mathbf{M}$ descends to a divisor $M_{X'}$ and $E$ is a divisor on $X'$.
Then $\operatorname{ord}_{E\slash X} (\mathbf{M})=\operatorname{coeff}_E (\varphi^*M-M_{X'})$.
It's clear that this definition does not depend on the choice of $X'$.
\end{defi}

\begin{rmk}\label{ordX/Z}
    Let $\phi:X\to Z$ be a birational morphism of normal varieties.
    Let $\mathbf{M}$ be a b-divisor over $Z$ 
    Assume that $\mathbf{M}$ descends to some divisor $L$ on some birational model $Y$ over $Z$.
    Replacing $Y$ by some common model $W$ over $X$ and $Y$ and replacing $L$ by $L|_W$, we may assume that $L$ is a divisor over $X$.  
    Then $\mathbf{M}$ can be view as a b-divisor over $X$ by viewing $\mathbf{M}$ as $[L]$.

    Moreover, if $\mathbf{M}|_Z$ is $\mathbb{R}$-Cartier, 
    then for any divisor $E$ over $X$,
    we have $\operatorname{ord}_{E/X}(\mathbf{M})-
    \operatorname{ord}_E(\mathbf{M}|_X)
    =\operatorname{ord}_{E/Z}(\mathbf{M})-\operatorname{ord}_E(\mathbf{M}|_Z)$.
    \end{rmk}

\begin{defi}(Multiplicity of divisors and b-divisor over $X$)
Let $X$ be a variety and let $\mathbf{M}$ be a b-divisor over $X$ such that its image $M=\mathbf{M}|_X$ on $X$ is a $\mathbb{R}$-Cartier divisor.
For a integral subvariety $\eta$ of $X$ with a smooth extraction (i.e. with its generic point in the smooth locus of $X$), we define the \emph{multiplicity} $\operatorname{mult}_{\eta} (\mathbf{M})$ of $\mathbf{M}$ at $\eta$ by letting $\operatorname{mult}_{\eta} (\mathbf{M})=\operatorname{ord}_{E_\eta\slash X} (\mathbf{M})$, where $E_\eta$ is the dominating exceptional divisor of the blow-up of $X$ along $\eta$.
Since blow-ups can be defined locally, $\operatorname{mult}_{\eta}$ is preserved if we replace $\eta$ and $X$ by their restrictions on any neighborhood of the generic point of $\eta$.

Consider integral subvarieties $\eta_1\subseteq\eta_2$ of $X$ with a smooth extraction (i.e. $\eta_2$ and $X$ are smooth at aroung the generic point of $\eta_1$). Assume that $\mathbf{M}$ is b-nef.
Replacing $\eta_1$, $\eta_2$, and $X$ by their restriction on an open neighbourhood of the generic point of $\eta_1$, we may assume that $\eta_1$, $\eta_2$, and $X$ are smooth.
Let $X_{\eta_1}$ be the blow-up of $X$ along $\eta_1$.
We inductively define the \emph{multiplicity} $\operatorname{mult}_{\eta_1\subseteq\eta_2} (\mathbf{M})$ of $\mathbf{M}$ along the extraction of $\eta_1\subseteq\eta_2$
by letting $\operatorname{mult}_{\eta_1\subseteq\eta_2} (\mathbf{M})=\operatorname{mult}_{\eta_1} (\mathbf{M})+\operatorname{mult}_{\eta'_1\subseteq\eta'_2} (\mathbf{M})$, where $\eta'_1$ is the extracted subvariety of $\eta_1\subseteq\eta_2$ and $\eta'_2$ is the strict transform of $\eta_2$ on $X_{\eta_1}$.
Alternatively,
we may define $\operatorname{mult}_{\eta_1\subseteq\eta_2} (\mathbf{M})$ in the following way.
Set $\eta_1^{(0)}=\eta_1$, $\eta_2^{(0)}=\eta_2$ and $X_0=X$ and define inductively $\eta_1^{(i)}$ to be the extracted subvariety of $\eta_1^{(i-1)}\subseteq\eta_2^{(i-1)}$ and $\eta_2^{(i-1)}$ to be the strict transform of $\eta_2$ on $X_i$, the blow-up of $X_{i-1}$ along $\eta_1^{(i-1)}$ for $i\geq 1$.
Then $\operatorname{mult}_{\eta_1\subseteq\eta_2} (\mathbf{M})$ is defined to be the sum $\sum_{i=1}^\infty\operatornamewithlimits{mult}_{\eta_1^{(i-1)}}(\mathbf{M})$.

It's clear that $\operatorname{mult}_{\eta_1\subseteq\eta_2} (\mathbf{M})$ is either a non-negative real number or positive infinity.
We may observe from the above definition that under its notations $\operatorname{mult}_E(\mathbf{M})=0$ for Cartier divisors $E$ on $C$ and so that $\operatorname{mult}_{\eta\subseteq X}(\mathbf{M})=\operatorname{mult}_\eta(\mathbf{M})$ for any subvariety $\eta$ of $X$ with a smooth extraction.

When $x\in \eta$ is a closed point in a subvariety $\eta\subseteq X$, we denote $\operatorname{mult}_{\{x\}\subseteq \eta}(\mathbf{M})$ by $\operatorname{mult}_{x\in \eta}(\mathbf{M})$ for simplicity.
\end{defi}

Recall the \emph{negativity lemma}.
\begin{lem}[{\cite[1.1]{Sho93}}]\label{neglem}
Let $f:Y\to X$ be a projective birational contraction between normal quasi-projective varieties over $\mathbb{C}$. Let $D$ be an $\mathbb{R}$-Cartier divisor in $Y$ such that
\begin{itemize}
    \item $-D$ is nef over $X$,
    \item $f_{*}(D)\geq 0$.
\end{itemize}
Then $D\geq 0$.
\end{lem}
The following results can be obtained from Lemma \ref{neglem}.
\begin{lem}\label{negcurve}
    Let $X\xrightarrow{\phi} Z$ be a birational morphism between surfaces and let $C$ be an exceptional curve on $X$.
    Let $\mathbf{M}$ be a relatively b-nef over $Z$ b-$\mathbb{R}$-Cartier b-divisor over $X$.
    If $M_X.C<0$, then $\operatorname{mult}_C\mathbf{M}\geq 0$.
\end{lem}
\begin{proof}
Let $Y\xrightarrow{\psi}X$ be a birational model over $X$ such that $M_{Y}=\mathbf{M}|_{Y}$ is a nef $\mathbb{R}$-Cartier divisor over $Z$, and $\psi$ factors through the blow-up of $X$ along $C$.

Consider the divisor $D=\psi^{*}M_{X}-M_{Y}$ in $Y$, and let $G$ be an exceptional curve of $\phi\circ\psi$ over $Z$. Notice that if $G$ is exceptional over $X$, then we have $D.C=0$. Otherwise, we use the fact that $M_{Y}$ is nef over $Z$, and apply the projection formula to obtain $\psi^{*}M_{X}.G=M_{X}.C<0$. Hence, we have
\[
D.G=\psi^{*}M_{X}.G-M_{Y}.G\leq 0.
\]
This shows that $-D$ is nef over $Z$. On the other hand, we have 
\[
(\phi\circ\psi)_{*}(D)=\phi_{*}(\psi_{*}(\psi^{*}M_{X}-M_{Y}))= 0.
\]Therefore, we get $D\geq 0$ by Lemma \ref{neglem}.
\end{proof}
\begin{lem}\label{bddmult}
    Let $X\xrightarrow{\phi} Z$ and $\mathbf{M}$ be as in Lemma \ref{negcurve}.
    Let $\eta$ be an integral closed subvariety on $X$.
    If there are infinitely many integral exceptional curves $C$ over $Z$ on $X$ intersecting $\eta$ such that $\mathbf{M}_X.C\leq a$ for some $a\in\mathbb{R}$, then we have $\operatorname{mult}_{\eta}\mathbf{M}\leq a$.
\end{lem}
\begin{proof}
   Suppose that $\operatorname{mult}_{\eta}\mathbf{M} > a$. Consider the blow-up $X_\eta$ of $X$ along $\eta$. Then for each the strict transform $\tilde{C}$ of $C$, we have $\mathbf{M}_{X_\eta}\cdot \tilde{C}< 0$. By Lemma \ref{negcurve}, $\operatorname{mult}_C\mathbf{M}>0$, for infinitely many curves $C$, which is impossible.
\end{proof}

We introduce the concept of \emph{potentially nef divisors}, which is crucial through the whole argument of our result.
\begin{defi}[Potentially Nef/Anti-Nef Divisors]
Let $X,Y$ be varieties, and let $\phi: Y\to X$ be a birational morphism between $X,Y$. An effective divisor $M_{Y}$ in $Y$ is \emph{potentially nef (potentially anti-nef, resp.)} over $X$ if there exists a variety $Z$, a birational morphism $\psi: Z\to Y$, and an effective divisor $M_{Z}$ in $Z$ which is nef (anti-nef, resp.) over $X$ such that $M_{Z}|_{Y}=M_{Y}$. In this case, we say $M_{Y}$ is a potentially nef (potentially anti-nef, resp.) divisor \emph{realized} by $M_{Z}$.  
\end{defi}

\begin{ex}
Let $X$ be a smooth threefold and let $x\in X$ be a closed point. Let $Y=\mathrm{Bl}_{x}(X)$ be the blow-up of $X$ at $x$ with the exceptional divisor $E_{x}$. Let $y,z,w\in E_{x}$ be three distinct closed points, and let $Z$ be the blow-up of $Y$ at $y,z,w$ with the exceptional divisors $E_{y},E_{z},E_{w}$, respectively. Then
\begin{itemize}
    \item $\tilde{E}_{x}+E_{y}+E_{z}$ is potentially anti-nef over $X$. Notice that it is not anti-nef over $X$.
    \item $\tilde{E}_{x}+E_{y}+E_{z}+E_w$ is potentially anti-nef over $X$ if and only if $x,y,z$ are collinear.
    \item Let $a$ be a real number. Then, $\tilde{E}_{x}+aE_{y}$ is potentially anti-nef over $X$ if and only if it is anti-nef if and only if $0\leq a\leq 1$.
\end{itemize}
\end{ex}

\begin{lem}\label{pot_nef}
Let $(X,B)$ be a $3$-dimensional pair and $f:Y\to X$ be a birational model over $X$. Then, there is a birational model $Z\to Y$ over $Y$ and an exceptional divisor $M_Z$ on $Z$ over $X$ such that $-M_Z$ is relatively nef over $X$ and $a_E(Z,[M_Z])=a_E(X,B)$ for every exceptional divisor $E$ on $Z$ over $X$. 

In other words, the strict transform $\tilde{B}$ of $B$ on $Y$ is potentially nef over $X$ realized by $B|_Z-M_Z$.
\end{lem}
\begin{proof}
This is an algorithmic construction. Let $W\to Y$ be a log resolution of $(X,B)$ such that the strict transform $B_W$ of $B$ in $W$ is smooth, and let $E_W=B|_W-B_W$ be the exceptional part of $B|_{W}$ over 
$X$. Write $B_W=\sum_{k=1}^Sb_k{B_W}_k$ for disjoint prime divisors ${B_W}_k$ in $W$ with $b_k>0$ (this is possible since we require $B_{W}$ to be smooth in $W$). Denote $\phi:Z\to X$ as the composite morphism of $W\to Y$ and $Y\to X$.

If $C$ is an exceptional curve over $X$ in $W$ with $C\cdot E_W=-C\cdot B_W>0$, then $C\subseteq B_W\cap E_W$. For $1\leq j\leq S$, write ${B_W}_{j}\cap E_W=\bigcup_{i=1}^{T_{j}}C_{j,i}$, where $C_{j,i}$ are distinct irreducible curves in $B_{W_{j}}\cap E_{W}$. Let $a=\max_{j,i}\{C_{j,i}\cdot E_W \}$. Let $H$ be a very ample divisor on $X$ and write $(\phi^*\phi_*H-H)\cdot B_W=\sum_{j,i} h_{ji} C_{j,i}$ as a $1$-cycle. 

 For $j,s\in [1,S]$, $i\in[1,T_j]$, $t\in [1,T_s]$ and $l\in [1,h_{ji}\lceil\frac{a}{b_{j}}\rceil]$, we construct birational models $Z_{j,i,l}$ over $Z_{1,1,0}:=W$ for $j,s\in [1,S]$, $i\in[1,T_j]$, $t\in [1,T_s]$ and $l\in [1,h_{ji}\lceil\frac{a}{b_{j}}\rceil]$, curves $C_{s,t,j,i,l}$ on $Z_{j,i,l}$, and exceptional$/W$ divisors $E_{j,i,l}$ and $M_{Z_{j,i,l}}$ on $Z_{j,i,l}$ as follows. Let $C_{j,i,1,1,0}:=C_{j,i}$ and $M_{Z_{1,1,0}}:=E_W$. 

For $l<h_{ji}\lceil\frac{a}{b_{j}}\rceil$, define $\phi_{j,i,l+1}:Z_{j,i,l+1}:=Bl_{C_{j,i,j,i,l}}Z_{j,i,l}\to Z_{j,i,l}$, $E_{j,i,l+1}:= \operatorname{Exc}(\phi_{j,i,l+1})$, $C_{j,i,j,i,l+1}:=E_{j,i,l+1}\cap \tilde{{B_W}_j}$, $C_{s,t,j,i,l+1}$ be the strict transform of $C_{s,t,j,i,l}$ for $(s,t)\neq(j,i)$, and $M_{Z_{j,i,l+1}}:=\phi_{j,i,l+1}^*M_{Z_{j,i,l}}+b_jE_{j,i,l+1}=B|_{Z_{j,i,l}}-\tilde{B}$, where $\tilde{{B_W}_j}$ and $\tilde{B}$ are the strict transforms of ${B_W}_j$ and $B$ respectively.

For $l=h_{ji}\lceil\frac{a}{b_{j}}\rceil$ and $i<T_j$, define $\phi_{j,i+1,1}:Z_{j,i+1,1}:=Bl_{C_{j,i+1,j,i,l}}Z_{j,i,l}\to Z_{j,i,l}$, $E_{j,i+1,1}:=  \operatorname{Exc}(\phi_{j,i+1,1})$, $C_{j,i+1,j,i+1,1}:=E_{j,i+1,1}\cap \tilde{{B_W}_j}$, $C_{s,t,j,i+1,1}$ be the strict transform of $C_{s,t,j,i,l}$ for $(s,t)\neq(j,i+1)$, and $M_{Z_{j,i+1,1}}:=\phi_{j,i+1,1}^*M_{Z_{j,i,l}}+b_jE_{j,i+1,1}=B|_{Z_{j,i+1,1}}-\tilde{B}$, where $\tilde{{B_W}_j}$ and $\tilde{B}$ are the strict transforms of ${B_W}_j$ and $B$ respectively.

For $l=h_{ji}\lceil\frac{a}{b_{j}}\rceil$, $i=T_j$ and $j<S$, define $\phi_{j+1,1,1}:Z_{j,i,l}:=Bl_{C_{j+1,1,j,i,l}}Z_{j,i,l}\to Z_{j,i,l}$, $E_{j+1,1,1}:=  \operatorname{Exc}(\phi_{j+1,1,1})$, $C_{j+1,1,j+1,1,1}:=E_{j+1,1,1}\cap \tilde{{B_W}_j}$, $C_{s,t,j+1,1,1}$ be the strict transform of $C_{s,t,j,i,j}$ for $(s,t)\neq(j+1,1)$, and $M_{Z_{j+1,1,1}}:=\phi_{j+1,1,1}^*M_{Z_{j,i,l}}+b_{j+1}E_{j+1,1,1}=B|_{Z_{j+1,1,1}}-\tilde{B}$, where $\tilde{{B_W}_{j+1}}$ and $\tilde{B}$ are the strict transforms of ${B_W}_{j+1}$ and $B$ respectively.




Let $Z:=Z_{S,T_S,h_{ST_S}\lceil\frac{a}{b_{S}}\rceil}$ and $M_Z:=M_{Z_{S,T_S,h_{ST_S}\lceil\frac{a}{b_{S}}\rceil}}$.
We claim that $-M_{Z}$ is nef. Suppose that $C$ is an exceptional curve on $Z$ such that $M_Z\cdot C>0$ with $c_W(C)=C_{ji}={B_W}_j\cap E$ for some exceptional prime divisor $E$ on $W$. Then, by the projection formula, we have $M_Z\cdot C=(-\sum_{j,i}\lceil a/b_j\rceil b_j h_{ji}C_{ji}+E_Z\cdot {B_W}_j)\cdot_{{B_W}_j} C_{ji}=\lceil a/b_j\rceil b_j  (\phi^*\phi_*H-H)\cdot C_{ji} +E_Z\cdot C_{ji}\leq -\lceil a/b_j\rceil b_j+a\leq 0$, a contradiction.
\end{proof}

\begin{defi}(Tower of Point-blow-ups)\label{buptower}
Let $X$ be a variety and let $n$ be a natural number.
A \emph{tower of point-blow-ups on $X$ of length $n$} is a sequence of morphisms $X_n\to X_{n-1}\to\dots\to X_1\to X_0=X$ satisfying
\begin{itemize}
\item $X_i\to X_{i-1}$ is the blow-up of $X_{i-1}$ at a closed point $x_{i-1}\in X_{i-1}$ for $1\leq i\leq n$,
\item the exceptional divisor $E_i$ of $X_i\to X_{i-1}$ contains $x_i$ for $1\leq i\leq n-1$
\end{itemize}
A tower of point-blow-ups on $X$ of length $n$ is called a \emph{chain} if additionally, the strict transform $\tilde{E}_{i-1}$ of $E_{i-1}$ on $X_i$ does not contain $x_i$ for $2\leq i\leq n$.
\end{defi}

\begin{lem}\label{longchain} 
Fix natural numbers $k$ and $N$, a positive real number $\delta$ and a bounded family of varieties $\mathcal{P}$, and consider the following senario:\\

Let $X\to Z$ be a birational morphism between 3-folds which has
\begin{itemize}
    \item an exceptional divisor $F\in\mathcal{P}$ with $\dim F_Z=0$,
    \item a relatively very-ample exceptional divisor $H$ over $Z$ with coefficients at most $N$, and
    \item a relatively b-nef over $Z$ b-$\mathbb{R}$-Cartier b-divisor $\text{M}$ over $Z$ with coefficients at least $\delta$, and $\operatorname{ord}_{F_h\slash Z}M\leq N$ for each irreducible components $F_h$ of $F$. 
\end{itemize}
Then there is a natural number $n\geq k$, only depending on $k,N,\delta,\mathcal{P}$, such that the following holds:\\

Given $X_n\to X_{n-1}\to\dots\to X_1\to X_0=X$
be the chain of blow-ups of as in Definition \ref{buptower} with exceptional divisors $E_i\subseteq X_i$ and centers $x_i\in X_i$.
If $D$ is a component of $F$ with $x_n\in \tilde{D}$ (where $\tilde{D}$ is the strict transform of $D$ on $X_{n}$),
then 
\begin{itemize}
    \item there is a curve $C_0\subseteq D$ such that $x_k\in \tilde{C}_0$ (where $\tilde{C_{0}}$ is the strict transform of $C$ on $X_{k}$), and
    \item $\operatorname{mult}_{C_{0}}\mathbf{M}=\operatorname{mult}_{x_i}M_i$.

\end{itemize}

\end{lem}

\begin{proof}
    Let $C=\sum_{j\in J} C_j$ be the sum of irreducible curves in the center of $M$.
    We may assume that $C\in\mathcal{P}$.
    There is a fixed natural number $t$ such that there are infinitely many integral curves $\Bar{C}\subseteq D$ with $p\in\Bar{C}$ and $\Bar{C}.M_X\leq t$ for any closed point $p\in D$.
    Replacing $M$ by $M+lH$ for some fixed natural number $l$, we may assume that $M_X$ is nef over $Z$ and $C_j.M_X\geq kt$ for all $j\in J$.
    Replacing $X$ by $X_s$ for some fixed natural number $s$,
    we may assume that $C$ is smooth at $x_0$.
    There is a fixed natural number $L$ such that $DM_X^2\leq L$.
    Since $\tilde{D}M_{X_n}^2\leq L-n\delta$,
    $M_{X_n}$ is not relatively nef over $Z$ if we pick $n>\frac{L}{\delta}$.
    This implies that $\tilde{C_j}M_{X_n}<0$ for some $j\in J$.
    So $x_k\in \tilde{C_j}$.
    Take $C_0=C_j$ and we have the first assertion.
    
    Let  $a=\operatorname{mult}_{x_i}M_i$. Assume that $\operatorname{mult}_{C_0}M<a$.
    Then we have $a-\operatorname{mult}_{C_0}\geq \delta$.
    Let $X_{C_0}$ be the blow-up of $X$ along $C_0$ with exceptional divisor $E_{C_0}$ and $x_{C_0}$ be its fiber over $x$.
    Then we have $\operatorname{mult}_{x_{C_0}\subseteq E_{C_0}}\geq k\delta$.
    We may assume $E_{C_0}\in\mathcal{P}$.
    Replacing $k$ and we get a contradiction.
\end{proof}

\begin{lem}\label{boundedcenter}
Fix a positive real number $\alpha$.

Suppose that $\phi:Y\to X$ is a birational morphism of projective normal threefolds with exceptional locus $\operatorname{Exc}(\phi)=E$ being a divisor such that $Y$ is smooth and $\dim \phi(E)=0$. Suppose that
$F$ is a smooth prime exceptional divisor of $\phi$ and $h$ is a very ample divisor on $F$. Suppose that $M$ is a divisor on $Y$ and $g$ is a curve on $F$ with $g\cdot_F h> \frac{h\cdot M}{\alpha}$. Let $\phi_g:Y_g\to Y $ be the blow-up of $Y$ along $g$ with exceptional divisor $E_g$. Then $\phi_g^*(M)-\alpha^+E_g$ is not potentially nef over $X$ for every real number $\alpha^+ \geq \alpha$.

\end{lem} 
\begin{proof} Suppose that $\alpha^+>\alpha$, then we have 
\[
\tilde{h}\cdot (\phi_g^*(M)-\alpha^+E_g)=h\cdot M-\alpha^+ E_g\cdot \tilde{h}= h\cdot M-\alpha^{+} g\cdot_F h<0.
\]Since $h$ is general, we have $\operatorname{mult}_h(\phi_g^*(M)-\alpha^+E_g)\geq 0$ by Lemma \ref{negcurve}, which implies that $\phi_g^*(M)-\alpha^+E_g$ is not potentially nef.
\end{proof}

\begin{lem}\label{boundedcurves}
Fix a DCC set $I\subseteq \mathbb{R}$, a positive real number $\alpha$ and a natural number $N$. Then there is a natural number $N'$ depending only on $I$, $\alpha$ and $N$ satisfying the following.

If $(X,B)$ is a pair of dimension $3$, such that
\begin{itemize}
\item there exists a log resolution $f:X\to X_0$ with $D=\operatorname{Exc}(f)$ being the $f$-exceptional divisor,
\item $\dim f(D)=0$,
\item $B$ and $D$ have no common components,
    \item $\operatorname{coeff} B\in I$,
    \item $\operatorname{coeff}f^*f_*B\leq N $,
    \item there is a very ample divisor $H\subseteq X$ such that $G:=f^*f_*(H)-H$ satisfies $-G^3\leq N$, and
    \item $C$ is an $f$-exceptional curve on $X$ with $C\cdot B<0$,
\end{itemize}
then 
\begin{enumerate}
    \item $C\cdot H\leq N'$ and
    \item $\operatorname{mult}_{C\subseteq D}B\leq N'$.
    \end{enumerate}
    
Assume, moreover, $X_N\to X_{N-1}\to\dots X_1:=X$ is a chain of point-blow-ups as in Definition \ref{buptower} with blow-up centers $x_i\in \tilde{D}\subseteq X_i$ and $\operatorname{mult}_{x_i}B=\alpha$. Then there is a curve $C'\subseteq D$ with $\operatorname{mult}_{C'}B=\alpha$ and $x_N\in \tilde{C}'$.

\end{lem}
\begin{proof}
Since $G$ is exceptional, $-G^3=G\cdot H^2$ is bounded. Therefore, $\operatorname{Exc}(f)$ is in a fixed bounded family. Since $B=f_*\tilde{B}$,by Lemma \ref{pot_nef} $B$ is potentially nef over $X_0$. By assumption, $f^*f_* B-B$ is bounded and potentially anti-nef over $X_0$. Since $C\cdot B<0$, we have 
$\operatorname{mult}_C B>0$. Let $H'$ be a general member of $|H|$. Then $(H'\cap G)\cdot B$ is bounded above. By Lemma \ref{boundedcenter}, (1) and (2) hold. The rest of the lemma follows from Lemma \ref{longchain}.
\end{proof}

\begin{defnlemma}\label{towerofsection}
Let $W$ be a threefold and $w\in W$ be a non-singular closed point and $X_0\to W$ be the blow-up at $w$ with exceptional divisor $E_0$. Let $M_0=E_0$. Let $l_0$ be a line in $E_0\cong \mathbb{P}^2$. Inductively, define $f_{i+1}:X_{i+1}\to X_i$ as the blow-up of $X_i$ along $l_i$ with the exceptional divisor $E_{i+1}$. Then, $E_i\simeq \mathbb{F}_{i+1}$ for $i\geq 1$ and there is a degree-$1$ section $l_{i+1}$ of $f_{i+1}|_{E_{i+1}}$ with $l_{i+1}\cdot \tilde{E}_i=0$, where $\mathbb{F}_{i}$ denotes the Hirzebruch surface associated to the sheaf $\mathcal{O}\oplus\mathcal{O}(-i)$. 

Let $M_{i+1}=M_i|_{X_{i+1}}+E_{i+1}$. 
Then $M_{i+1}$ is anti-nef over $W$, $l_{i+1}\cdot M_{i+1}=-(i+2)$, $l_{i+1}\cdot_{E_{i+1}}l_{i+1}=i+2$ and for every degree-$1$ section $s\not\subseteq E_i$ of $f_{i+1}|_{E_{i+1}}$, $s\cdot M\leq -(i+2)$.

Suppose $M_{i+1}$ is realized by an anti-nef divisor $M_Z$ as a potentially anti-nef divisor over $W$, then $\operatorname{mult}_p M_Z=\operatorname{mult}_p M_{i+1}$ for every closed point $p\in E_i$ in $X$.

$E_i^3=1-i$, $E_iM_i^2=i+1$, $E_0|_{X_i}M_i^2=\sum_{j=0}^i\tilde{E}_iM_i^2=i+1$.
 
 Define \emph{lines} on $E_i$ as sections of $f_{i}|_{E_{i}}$ not intersecting $\tilde{E}_{i-1}$. Let $l_{i-1}\subseteq E_i$ be the curve $E_i\cap \tilde{E}_{i-1}$ and $F$ be a fiber of $f_{i}|_{E_{i}}$. The any line is linearly equivalent to $l_{i-1}+(i+1)F$ on $E_i$. $F\cdot_{E_i} l_{i-1}=1$, $F\cdot_{E_i} F=0$, $l_{i-1}\cdot_{E_i} l_{i-1}=-(i+1)$, $l_{i}\cdot_{E_i} l_{i}=i+1$.

In other words, $i+2$ points, counted with multiplicity, in $E_i\setminus\tilde{E}_{i-1}$ uniquely determine a degree-$1$ section of $f_{i}|_{E_{i}}$ not intersecting $\tilde{E}_{i-1}$.

For $n\in \mathbb{N}\cup\{ 0\}$, the chain of blow-ups $X_n\to X_{n-1}\to \dots\to X_0\to W$ is called a \emph{tower of section-blow-ups of layer $n$} at $w$ with layers $E_0,E_1,\dots,E_n$.
\end{defnlemma}
\begin{proof}
This follows from the projection formula in intersection theory.
\end{proof}

\begin{defi}[Tower of Section-Blow-Ups]\label{multofsectionbup}
    Let $X$ be a smooth threefold, $B$ be an effective divisor on $X$ and $x\in X$ is a closed point. Let $k$ be a natural number. Let $\mathbf{a}=(a_0,a_1,\dots,a_k)$ be a $(k+1)$-tuple of non-negative real numbers. Then, a \emph{tower of section-blow-ups $Y\to X$ of layer $k$ with multiplicities $\mathbf{a}$ along $B$ at $x$} is a blow-up $Y_0\to X$ of $X$ at $x$ followed by a sequence of blow-ups $Y=Y_{k}\to Y_{k-1}\to\cdots\to Y_0$ of with blow-up centers $l_i\subseteq Y_i$ and exceptional divisors $F_i\in Y_i$ such that $l_i \in F_i$ is a line,  $\operatorname{mult}_x B=a_{0}$ and $\operatorname{mult}_{l_{i}}\tilde{B}=a_i$ for all $i\geq 1$.
    If $a_0=a_1=\dots=a_k=a$, then we may omit $\mathbf{a}$ and say that $Y\to X$ is with multiplicity $a$ along $B$ at $x$.
\end{defi}

\begin{lem}\label{mult}
Let $X$ be a variety and let $\mathbf{M}$ be a b-nef b-divisor over $X$ with image $M$ on $X$ an $\mathbb{R}$-Cartier divisor.
Let $\eta_1\subseteq\eta_2$ be integral subvarieties of $X$ such that the extraction of $\eta_1$ along $\eta_2$ is smooth.
Then $\operatorname{mult}_{\eta_1} (\mathbf{M})\geq\operatorname{mult}_{\eta_2} (\mathbf{M})$.
\end{lem}
\begin{proof}
Replacing $X$ by a neighbourhood of $\eta_1$, we may assume that $\eta_1$, $\eta_2$ and $X$ are smooth.
    Let $\varphi_1:X_{\eta_2}\to X$ be the blow-up of $X$ along $\eta_2$ with exceptional divisor $E_{\eta_2}$.
    Let $\varphi_2:X'\to X_{\eta_2}$ be the blow-up of $\varphi_1^{-1}(\eta_1)$ with exceptional divisor $E_{\varphi_1^{-1}(\eta_1)}$
    and $\varphi_{\eta_1}:X_{\eta_1}\to X$ be the blow-up of $X$ along $\eta_1$.
    Since the preimage of $\eta_1$ on $X'$ is a Cartier divisor, $X'\to X$ factor through $\eta_1$ and $E_{\varphi_1^{-1}(\eta_1)}$ dominates $E_{\eta_1}$.
    So $\operatorname{mult}_{\eta_1} (\mathbf{M})=\operatorname{ord}_{E_{\eta_1}/X} 
    (\mathbf{M})\geq \operatorname{ord}_{E_{\eta_2}/X}(\mathbf{M}) \operatorname{coeff}_{ E_{\varphi_1^{-1}(\eta_1)}}(\varphi_2^*E_{\eta_2})\geq\operatorname{mult}_{\eta_2}(\mathbf{M})$.
\end{proof}

\begin{lem}\label{lim}
    Let $B\subseteq A$ be sets and let
    $C=\big\{\{
    a_{ij}\}_{j=0}^\infty\big\}_{i=0}^\infty$ be a sequence of sequences of elements $a_{ij}$ of $A$.
    Then one of the followings holds.
    \begin{itemize}
        \item There is a subsequence $\big\{\{
    a_{k_ij}\}_{j=1}^\infty\big\}_{i=1}^\infty$ of $C$ and a natural number $n$ such that $\min\{ j|a_{k_ij}\in B\}=n$ for all $i$.
    \item There is a subsequence $\big\{\{
    a_{k_ij}\}_{j=0}^\infty\big\}_{i=0}^\infty$ of $C$ such that $\min\{ j|a_{k_ij}\in B\}\geq i$ for all $i$.
    \end{itemize}
\end{lem}
\begin{proof}
    If $\limsup_i{\min\{ j|a_{ij}\in B\}}=n$ for some natural number $n$, then the first statement holds.
    Otherwise, we have $\limsup_i{\min\{ j|a_{ij}\in B\}}=\infty$ and the second statement holds.
\end{proof}
\begin{lem}\label{B=0}
    Let $(X,B)$ be a pair of dimension $2$.
    Assume that $X$ is smooth and $X'$ is a birational model over $X$.
    Then there is a b-nef$/X$ exceptional b-$\mathbb{R}$-Cartier divisor $\mathbf{M}$ over $X$ with coefficients in the set of coefficients of $B$ such that
    $\operatorname{ord}_{F/X} (\mathbf{M})\leq \operatorname{ord}_{F/X} (B)$ for every divisor $F$ over $X$
and $\operatorname{ord}_{E/X}(\mathbf{M})= \operatorname{ord}_{E/X}(B)$ for every exceptional divisor $E$ on $X'$ over $X$.
\end{lem}
\begin{proof}
    Write $B=\sum b_iB_i$ for positive real numbers $b_i$ and prime divisors $B_i$. By possibly replacing $X'$ by a higher model, we may assume that $\phi:X'\to X$ is a log resolution of $(X,B)$.
    Let $\tilde{B_i}$ be the strict transform of $B_i$ on $X'$.
    
    Since all exceptional curves on $X'$ intersects non-negatively with $\tilde{B_i}$ for all $i$, $M':=\sum b_i(\tilde{B_i}-\phi^*B_i)$ is a nef$/X$ divisor on $X'$ with coefficients in the set of coefficients of $B$.
    Set $\mathbf{M}=[M']$.
    Notice that 
    \[
    \operatorname{ord}_{E/X}(\mathbf{M})=-\operatorname{coeff}_E(M')=\operatorname{coeff}_E(\phi^*B-\tilde{B})=\operatorname{ord}_{E/X}(B)
    \]for every exceptional divisor $E$ on $X'$ over $X$, and $(\mathbf{M}+[B])|_Y\geq 0$ for every birational model $Y$ over $X$, the lemma follows.
    \end{proof}
\begin{thm}\label{surf}
Fix a natural number $N$
and a positive real number $\delta$,
there exists a positive real number $\epsilon$ satisfying the following statement:
If $(X,B+\mathbf{M})/ Z$ is a generalized sub-lc pair of dimension $2$ such that
\begin{itemize}
    \item $X$ is smooth and projective,
    \item $\dim Z=2$,
    \item $B=B^+-B^-$ with $B^+\geq0$ and $B^-\geq 0$,
    \item $\mathbf{M}$ is a b-nef$/X$ b-$\mathbb{R}$-Cartier divisor over $X$ with $\mathbf{M}|_X=M$,
    \item $x\in X$, $X_x\to x$ is the blow-up of $X$ along $\overline{\{x\}}$ and $E_x$ is the exceptional divisor,
    \item $\mathbf{E}=[-E_x]$,
    \item $\operatorname{mult}_{x}(B^++\mathbf{M})\leq N$ 
    \item the coefficients of $B^+$ and $\mathbf{M}$ are at least $\delta$,
\end{itemize}
then, we have $\operatorname{mld}_x(X,B+\mathbf{M}+\epsilon m \mathbf{E})\geq 0$, where $m=\operatorname{mld}_x(X,B+\mathbf{M})$.
\end{thm}
\begin{proof}


    Assume the statement doesn't hold. Then there is a sequence of generalized sub-lc pairs $(X_i,B_i+\mathbf{M}_i)\to Z_i$, for $i=1,2,3,\dots$ with $B^+_i$, $B^-_i$, $\mathbf{M}_i$, $\mathbf{E}_i$, $x_i$ and $m_i$ satisfying the conditions in the statement such that $(X_i,B_i+ \mathbf{M}_i+\frac{1}{i} m_i \mathbf{E}_i)$ is not lc at $x_i$.
    
    Let $F_i$ be a divisor over $x_i$ such that $a_{F_i}(X_i,B_i+\mathbf{M}_i+\frac{1}{i} m_i \mathbf{E}_i)<0$.
    Since $\operatorname{mld}_{c_X(F_i)}(X_i,B_i+\mathbf{M_i})\geq m_i$, we may replace $x_i$ and assume it is the center of $F_i$.
    
    By Lemma \ref{B=0}, for each $i$,
    there is a b-nef$/X_i$ b-$\mathbb{R}$-Cartier exceptional divisor$/Z_i$ $\mathbf{B}^+_i$ such that
    $\operatorname{ord}_{F/X_i}( \mathbf{B}^+_i)\leq \operatorname{ord}_{F} (B_i^+)$ for every divisor $F$ over $X$
    and $\operatorname{ord}_{F_i/X_i}(\mathbf{B}^+_i)= \operatorname{ord}_{F_i}(B_i^+)$.
    
    Replacing $\mathbf{M}_i$ by $\mathbf{M}_i+\mathbf{B}^+_i$, we may assume that $B_i^+=0$.

   
    Let $X_{ij}\to\dots\to X_{i1}\to X_{i0}:=X_i$ be the sequence of blow-ups along the centers of $F_i$. Denote the morphism $X_{ij}\to X_{i(j-1)}$ by $\varphi_{ij}$, the exceptional divisor by $F_{ij}$ and the blow-up center by $x_{i(j-1)}$.
    Possibly replacing $F_i$ by some $F_{ij}$, we may assume that $a_{F_{ij}}(X_i,B_i+\mathbf{M}_i+\frac{1}{i} m_i \mathbf{E}_i)\geq 0$ for all $F_{ij}\neq F_i$.
    Since $a_{F_i}(X_i,B_i+\mathbf{M}_i+\frac{1}{i} m_i \mathbf{E}_i)=a_{F_i}(X_i,B_i+\mathbf{M}_i)-\frac{1}{i} m_i\operatorname{ord}_{F_i/X_i}(\mathbf{E}_i)\geq m_i-\frac{1}{i} m_i\operatorname{ord}_{F_i/X_i}(\mathbf{E}_i)$, we have $\operatorname{ord}_{F_i/X_i}(\mathbf{E}_i)> i$.
    In particular, $\dim \overline{\{x_{i}\}}=0$ for every $i$.
     
    Since $(X_{ij},\sum_jF_{ij})$ has only simple normal crossing singularities,
    $\operatorname{ord}_{F_i/X_i}(\mathbf{E}_i)\leq 2^j$ for every $i$ and $j$.
    Passing to a subsequence, we may assume that $\dim \overline{\{x_{i(i+1)}\}}=0$ for every $i$.

    Let $E_{ij}$ be the image of $\mathbf{E}_i$ on $X_{ij}$ and $M_{ij}$ be the image of $M_i$
    on $X_{ij}$.
    Since $\operatorname{coeff}_{E_{i1}}M_{i1}\leq N$ by assumption, twe have $-E_{i1}M_{i1}\leq N$ for all $i$.
    
    
    For each $j$, let $k_j$ be the smallest natural number such that $E_{ik_j}=E_i$. By assumption above, we have $a_{E_{ik_j-1}}(X_i,B_i+\mathbf{M}_i)\geq 0$.
    This implies that 
    $\operatorname{mult}_{x_{i(k_j-1)}}(\mathbf{M}_i):=a_{ik_j-1}\geq\delta$.
    Therefore, $N\geq -E_{i1}M_{i1}=-E_{ik_j}M_{ik_j}\geq(\operatorname{coeff}_{F_{ik_j}}E_{ik_j})E_{ik_j}M_{ik_j}=-\operatorname{coeff}_{F_{ik_j}}(E_{ik_j})a_{ik_j-1}\geq i\delta$ for all $i$, which gives the desired contradiction.
\end{proof}
\begin{lem}\label{a}
    Let $X\to Z$ be birational surfaces and let $\mathbf{M}$ be a b-nef$/Z$ b-$\mathbb{R}$-Cartier b-divisor over $X$ with image $M$ on $X$.
   
    Suppose that $D$ is a prime divisor on $X$ that is exceptional over $Z$ and $x\in X$ is a closed point such that 
    $X$ and $D$ are smooth at $x$.
    
    Then we have $\operatorname{mult}_{x\in D}\mathbf{M}\leq M\cdot D$.
\end{lem}
\begin{proof}
    Let $\operatorname{mult}_{x}\mathbf{M}=a$.
    Let $X_1$ be the blow-up of $X$ along $x$, $E_1$ be the exceptional divisor, $D_1$ be the strict transform of $D$ on $X_1$, $x_1$ be the intersection $E_1\cap D_1$ and $M_1$ be the image of $\mathbf{M}$ on $X_1$. Then $\operatorname{mult}_{x_1\in D_1 }\mathbf{M}=\operatorname{mult}_{x\in D}\mathbf{M}-a$, $M_1\cdot D_1=M\cdot D-a$ and $X_1$ and $D_1$ are smooth at $x_1$.
    Inductively replacing $X$, $x$, $D$ and $M$ by $X_1$, $x_1$, $D_1$ and $M_1$ respectively, we may assume that $a=0$.
    This implies that $\operatorname{mult}_{x\in D}\mathbf{M}=0$.
    
    On the other hand, since $D\geq 0$ is exceptional over $Z$ and $M$ is nef over $Z$, we have $M\cdot D\geq 0$.
    \end{proof}

\begin{lem}
Let $f:Y\to X$ and $g:Z\to Y$ be birational morphisms and $M_Y\subseteq Y$ be a potentially nef divisor over $X$ realized by $M_Z\subseteq Z$. Suppose that $c$ is an $f$-exceptional curve with $c\cdot M_Y<0$. Then $\operatorname{mult}_cM_Y>\operatorname{mult}_cM_Z$. That is, $\operatorname{mult}_c[M_Z]>0$.
\end{lem}
\begin{proof}
Replacing $Z$ by a higher model and replacing $M_Z$ by its pullback accordingly, we may assume that $g$ factors through the blow-up $\phi_c:Y_c\to Y$ along $c$.

Notice that $g^*(M_Y)-M_Z=g^*g_*(M_Z)-M_Z$ is $g$-exceptional and $g$-anti-nef, so it is effective by the negativity lemma \ref{neglem}, which implies
$\operatorname{mult}_{c}M_{Y}\geq\operatorname{mult}_{c}M_{Z}$.

We claim that $\operatorname{mult}_{c}M_{Y}>\operatorname{mult}_{c}M_{Z}$. Assume that $\operatorname{mult}_cM_Y=\operatorname{mult}_cM_Z$.
Then $M_Z|_{Y_c}=M_Y|_{Y_c}$.

Let $E_c$ be the exceptional divisor of $\phi_c$ let $c'$ be a general section of $\phi_c|_E$ with degree $d$. As $c'$ is taken generally, we may assume that $g$ is isomorphic over the generic point of $c'$.
Let $c_Z$ be the strict transform of $c'$ on $Z$. Since $M_Z|_{Y_c}=M_Y|_{Y_c}$, $\operatorname{ord}_{E_c}(g^*M_Y-M_Z)=0$ and therefore $c_Z$ is not in the support of $g^*(M_Y)-M_Z$. So $0>d(c\cdot M_Y)=c_Z\cdot g^*(M_Y)\geq c_Z\cdot M_Z$, contradicting the nefness of $M_Z$ over $X$.
\end{proof}

\begin{lem}\label{2.35}
Let $f:Y\to X$ and $g:Z\to Y$ be birational morphisms, and $M_Y\subseteq Y$ be a potentially nef divisor over $X$ realized by $M_Z\subseteq Z$. Let $\eta$ be a subvariety of $Y$, and let that $c$ is an $f$-exceptional curve which is not contained in $\eta$, $M_Y|_{X_c}=M_Z|_{X_c}$, and $|\{c\cap\eta\}|=n$, where $X_c$ is the blow-up of $X$ along $c$. Let $\phi_\eta:Y_\eta\to Y$ be the blow-up of $Y$ along $\eta$ and let $c_\eta$ be the strict transform of $c$ on $Y_\eta$.
Then $c_\eta\cdot M_Z|_{Y_\eta}-n\operatorname{mult}_\eta M_Z\leq c\cdot M_Y-n\operatorname{mult}_\eta M_Y$. That is, $c_\eta\cdot M_Z|_{Y_\eta}\leq c\cdot M_Y-n\operatorname{mult}_\eta ([M_Z])$.
\end{lem}
\begin{proof}
    This follows from $c\cdot M_Y=c_\eta\cdot M_Y|_{Y_\eta}$, $M_Y|_{Y_\eta}-M_Z|_{Y_\eta}=(\operatorname{mult}_\eta M_Y-\operatorname{mult}_\eta M_Z)E_\eta$, and $c_\eta\cdot E_\eta\geq n$, where $E_\eta$ is the exceptional divisor of $\phi_\eta$.
\end{proof}

\begin{lem}\label{bup1}
    Let $X$ be a smooth threefold, $D\subseteq X$ be a smooth surface, $C\subseteq X$ be a smooth curve such that $C$ and $D$ intersect transversely at a closed point $x$. Let $X_C\to X$ be the blow-up of $X$ along $C$ with the exceptional divisor $E_C$ and $f$ be the fiber over $x$ of the restriction $E_C\to C$. Let $X_f\to X_C$ be the blow-up of $X_C$ along $f$ with the exceptional divisor $E_f$ and $\ell:=E_f\cap\tilde{D}$. Let $X_{\ell}\to X_f$ be the blow-up of $X_f$ along $\ell$ with the exceptional divisor $E_{\ell}$. On the other hand, let $X_x\to X$ be the blow-up of $X$ along $x$ with the exceptional divisor $E_x$ and $h:=E_x\cap \tilde{D}$. Let $X_h\to X_{x}$ be the blow-up of $X_x$ along $h$ with the exceptional divisor $E_h$. 

    Then
    \begin{enumerate}
        \item $E_f=E_x$ as prime divisors over $X$, and
        \item $E_{\ell}=E_h$ as prime divisors over $X$.
    \end{enumerate}
\end{lem}
\begin{proof}
    We observe from the universal property of blow-ups that the induced map from $X_F$ to $X_x$ is the blow-up of $X_x$ along $\tilde{C}$ and $\tilde{h}=\ell$, and the lemma follows.
\end{proof}
Similarly to Lemma \ref{bup1}, we have the following lemma.
\begin{lem}\label{bup}
    Let $X$ be a smooth threefold, $C\subseteq X$ be a smooth curve, and $x \in C$ be a closed point. Let $X_x$ be the blow-up of $X$ at $x$ with the exceptional divisor $E_x$. Let $X_C$ be the blow-up of $X$ along $C$ with the exceptional divisor $E_C$ and let $f$ be the fiber over $x$ of the restriction $E_C\to C$. 
    
    Suppose that $l\subseteq E_x\simeq \mathbb{P}^2$ is a line intersecting $\tilde{C}$, $X_l$ is the blow-up of $X_x$ along $l$ with the exceptional divisor $E_l$. Then there is a unique closed point $y\in E_C$ with $c_X(y)=x$ such that if $X_y$ is the blow-up of $E_C$ at $y$ with the exceptional divisor $E_y$, then the followings hold.
    
    \begin{enumerate}
        \item $E_l=E_y$ as prime divisors over $X$.
        \item For any closed point $p\in E_l\setminus \tilde{E}_x$ with $c_{E_x}(p)\not\in \tilde{C}$, there is a closed point $q\in E_y\setminus\tilde{E}_C$ such that $p\in X_l$ is isomorphic to $q\in X_y$ locally around $p$ and $q$, respectively.
    \end{enumerate}
\end{lem}

Applying Lemma \ref{bup} inductively, we have the following lemma.
\begin{lem}\label{primedivisors}
    Let $X$ be a smooth threefold, $D\subseteq X$ be a smooth surface, $C\subseteq D$ be a smooth curve and $x\in C$ be a closed point.
    Let $n$ be a natural number.
    Let $Y_1\to Y_0:=X$ be the blow-up $X$ along $C$ with the exceptional divisor $F_1$ and $C_1:=F_1\cap \tilde{D}$.
    Inductively define $Y_i\to Y_{i-1}$ as the blow-up of $C_{i-1}$ with the exceptional divisor $F_i$ and $C_i:=F_i\cap\tilde{D}$. Let $x_n$ be the close point on $C_n$ mapped onto $x$ and $Y'_n$ be the blow-up of $Y_n$ at $x_n$ with exceptional divisor $E_{x_n}$.
    
    Let $X_x$ be the blow-up of $X$ at $x$ with exceptional divisor $E_x$ and $l:=E_x\cap\tilde{D}$. Let $Z_1\to Z_0:=X_x$ be the blow-up of $X_x$ along $l$ with the exceptional divisor $G_1$ and $l_1:=G_1\cap\tilde{D}$.
    Inductively define $Z_i\to Z_{i-1}$ as the blow-up of $l_{i-1}$ with the exceptional divisor $G_i$ and $l_i:=G_i\cap\tilde{D}$. 
    
    Then, $G_n=E_{x_n}$ as prime divisors over $X$.
\end{lem}
\begin{lem}\label{bupmult}
     Let $x\in C\subseteq X$ be a closed point $x$ on a curve $C$ on a threefold $X$.
    Assume that $X$ and $C$ are smooth at $x$.
    Let $\varphi_C:X_C\to X$ be the blow-up of $X$ along $C$ with the exceptional divisor $E_C$ and fiber $C_x$ over $x$.
    Let $\varphi_{C_x}:X_{C_x}\to X_C$ be the blow-up of $X_C$ along $C_x$.
    Let $\varphi_x:X_x\to X$ be the blow-up of $X$ at $x$ with the exceptional divisor $E_x$ and let $\tilde{C}$be the strict transform of $C$ on $X_x$.
    Let $\varphi_{\tilde{C}}:X_{\tilde{C}}\to X_x$ be the blow-up of $X_x$ along $\tilde{C}$.
    Then by the universal property of blow-ups, there is an isomorphism between $X_{C_x}$ and $X_{\tilde{C}}$ under which $\varphi_x\circ\varphi_{\tilde{C}}$ and $\varphi_{C_x}\circ\varphi_C$ commute.
    By abuse of notation, we denote both $X_{C_x}$ and $X_{\tilde{C}}$ by $Y$, up to this isomorphism.
    Let $y$ be a closed point on $C_x$ and let $L$ be the preimage of $y$ on $Y$ under $\varphi_{C_x}$.
    Let $L_x$ be the image of $L$ on $X_x$.
    Then $L_x$ is a line on $E_x\simeq \mathbb{P}^2$.
    
    Let $\mathbf{M}$ be a b-$\mathbb{R}$-Cartier divisor which is b-nef over $X$.
    Then we have
    \begin{itemize}
        \item [(1)] $\operatorname{mult}_C\mathbf{M}+\operatorname{mult}_{C_x}\mathbf{M}=\operatorname{mult}_x\mathbf{M}$, and
        \item [(2)] $\operatorname{mult}_C\mathbf{M}+\operatorname{mult}_y\mathbf{M}=\operatorname{mult}_x\mathbf{M}+\operatorname{mult}_{L_x}\mathbf{M}.$
    \end{itemize}
\end{lem}

\begin{proof}
This follows inductively from Lemma \ref{bup}.
\end{proof}

\begin{lem}\label{induction_on_layer}
Fix a non-negative integer $n$.
Let $W$ be a threefold and let $w\in W$ be a non-singular closed point. Assume that $X_n\to X_{n-1}\to\dots\to X_0\to W$ is a tower of section-blow-ups of layer $n$ at $w$ with layers $M_0=E_0,E_1,\dots,E_n$, $x\in E_n\setminus\tilde{E}_{n-1}$, and $Y_n\to Y_{n-1}\to\dots\to Y_0\to X_n$ is a tower of section-blow-ups at $x$ of layer $n$ with layers $N_0=F_0,F_1,\dots,F_n$. Define $M_i$ and $N_i$ inductively for $0\leq i\leq n$ by $M_{i}=M_{i-1}|_{X_i}+E_{i-1}$ for $1\leq i\leq n$, $N_0=M_n|_{Y_0}+F_0$, and $N_{i}=N_{i-1}|_{Y_i}+F_{i-1}$ for $1\leq i\leq n$. Suppose $y_n\in F_n\cap \tilde{E}_n\setminus\tilde{F}_{n-1}\subseteq Y_n$ is a closed point and $Y'\to Y_n$ is the blow-up of $Y_n$ at $y_n$ with the exceptional divisor $F_y$. Let $0\leq a\leq 1$ be a real number. Then
\begin{enumerate}
    \item $N_n$ is anti-nef over $W$ if and only if $c_{Y_0}(F_1)\not\subseteq \tilde{E}_n$,
    \item $N_y:=N_n|_{Y'}+aF_y$ is not anti-nef over $W$,
    \item $N_y$ is potentially anti-nef over $W$,
    \item there is a unique exceptional curve $\ell$ over $W$ on $Y'$ such that $\ell\cdot N_y>0$,
    \item $c_{X_n}(\ell)$ is a line in $E_n\simeq \mathbb{F}_{n+1}$ with $\ell\cdot N_y=a$,
    \item if $N_n$ is realized by some nef divisor $N_Z$ over $W$ on some model $Z$ over $Y'$, then 
    \[
    \operatorname{mult}_{\ell\subseteq \tilde{E}_n}Z\geq\operatorname{mult}_{\ell}N_y+a=a
    \]

\end{enumerate}
\end{lem}
\begin{proof}
(1)-(5) follow from computations in Lemma \ref{towerofsection} and the projection formula in intersection theory. (6) follows from Lemma \ref{basic}.
\end{proof}

\begin{lem}\label{strict_transform}
    Let $W$ be a smooth threefold and $w\in W$ be a closed point. Let $B$ be an effective divisor on $W$. Let $a$ be a positive real number and $n$ be a non-negative integer. Assume that $Y_n\to Y_{n-1}\to\dots\to Y_0\to X_n\to X_{n-1}\to\dots\to X_0\to W$ and $E_0,E_1,\dots,E_n,F_0,F_1,\dots,F_n$ are given as in Lemma \ref{induction_on_layer}. Suppose that $\operatorname{ord}_{E_i}B=(i+1)a$ for $0\leq i\leq n$ (i.e. $\operatorname{mult}_{c_{X_{i-1}}}\tilde{B}=a$, where $X_{-1}:=W$). Then the followings hold.
    
 \begin{itemize}
        \item
    There are finitely many lines $l_j$ in $E_n\simeq \mathbb{F}_{n+1}$ such that on $X_n$, $a_j:=\operatorname{mult}_{l_i}\tilde{B}=\operatorname{mult}_{l_i\subseteq E_n}\tilde{B}$ and $\sum_j a_j= a$. Moreover, $\operatorname{mult}_{p}\tilde{B}=0$ for $p\in E_n$ outside $\bigcup\{l_j\}$

    \item If, in addition, $\operatorname{ord}_{F_i}B=(n+2+i)a$ for $0\leq i\leq n$ (i.e. $\operatorname{mult}_{c_{Y_{i-1}}}\tilde{B}=a$, where $Y_{-1}:=X_n$), then $c_{X_{n-1}}(F_n)\tilde{l_j}$ for all $j$ with $p_j:=F_n\cap\tilde{l_j}\subseteq Y_n$ distinct and $\operatorname{mult}_{p_j}\tilde{B}=a_j$. Therefore, every two lines $\{l_j\}$ intersect exactly at $x$ with multiplicity $n+1$.
    In particular, there is a curve $c\subseteq F_n$ with $\operatorname{mult}_c\tilde{B}=a$ only if there is a line $c'\subseteq E_n\simeq \mathbb{F}_{n+1}$ with $\operatorname{mult}_c\tilde{B}=a$.
    \item If, in addition, there is a closed point $p\in F_n\cap\tilde{E}_n\subseteq Y_n$ with $\operatorname{mult}_p\tilde{B}=a$, then there is a line $c'\subseteq E_n\simeq \mathbb{F}_{n+1}$ with $\operatorname{mult}_c\tilde{B}=a$.
    \end{itemize}
\end{lem}
\begin{proof}
    This follows from the projection formula in intersection theory.
\end{proof}

\begin{lem}\label{sectionmult1}
    Under the notations of Definition \ref{multofsectionbup}. Let $Y=Y_{k}\to Y_{k-1}\to\cdots\to Y_0\to X$ be a tower of section-blow-ups of layer $k$ with multiplicity $\mathbf{a}=(a_0,a_1,\dots,a_k) $ along $B$ with blow-up centers $x\in X$, $l_i\subseteq Y_i$ and layers $F_i\subseteq Y_i$.
Suppose $y\in F_k$ is a closed point with and $Z=Z_{k}\to Z_{k-1}\to\cdots\to Z_0\to Y$ is a tower of section-blow-ups of layer $k$ with multiplicity $\mathbf{a}$ along $B$ with blow-up centers $y\in Y$, $h_i\subseteq Z_i$ and layers $G_i\subseteq Z_i$. Then $\tilde{B}\cdot F_k=\sum d_j s_j$ for distinct lines $s_j$ on $F_k\simeq \mathbb{F}_{k+1}$. Moreover, $s_j$'s mutually intersect with each other on $F_k\simeq \mathbb{F}_{k+1}$ at $y$ with multiplicity $k+1$ and $\tilde{B}\cdot G_k=\sum d_j s'_j$ for degree-$1$ sections $s'_j$ on $G_k\simeq \mathbb{F}_{k+1}$ such that $s'_j$ intersects $\tilde{s_j}$.
\end{lem}
\begin{proof}
    This follows from Lemma \ref{induction_on_layer} and Lemma \ref{strict_transform}.
\end{proof}

\begin{lem}\label{sectionmult}
    Under the notation of Lemma \ref{sectionmult1}. Suppose that $C\subseteq X$ is a smooth curve such that $\operatorname{mult}_CB=c\leq a_k$ and $F_k$ intersects $\tilde{C}$ on $Y_k$. Let $Q_i$ be the exceptional divisor of $X_i\to X_{i-1}$ for $0\leq i\leq k-1$. Let $X_{C}$ be the blow-up of $X$ along $C$ with the exceptional divisor $E_{C}$. 
    Let $h_{-2}\subseteq X_C$ be the preimage of $x$ and $h_i=Q_i\cap \tilde{E_C}\subseteq X_i$ for $0\leq i\leq k-1$.
    Let $X_{i0}=X_i$, $Q_{i0}=Q_i$, and $h_{i0}=h_i$. Inductively define $X_{ij}$ be the blow-up of $X_{i(j-1)}$ along $h_{j-1}$ with the exceptional divisor $Q_{ij}$ and $h_{ij}=Q_{ij}\cap\tilde{E}_C$ for $j=1,2,3,\dots$.
    Let $Y_{\tilde{C}}$ be the blow-up of $Y_k$ along $\tilde{C}$ with the exceptional divisor $E_{\tilde{C}}$.
    Set $M=-\sum_{j=0}^k(1+j)\tilde{F}_j$ and $N=M-E_{\tilde{C}}$ as nef$/X$ divisors on $Y_{\tilde{C}}$, $\mathbf{N}=[N]$, and $\mathbf{M}=[M]$. 
    Let $Y_C\to Y_k$ be the blow-up along $\tilde{C}$. 
    Let $D=\tilde{B}-B|_{Y_C}$ and let $\mathbf{D}=[D]$ as a b-nef/$X$ b-divisor over $X$. Then the followings hold.
    \begin{enumerate}
        \item Starting from $X_C$, blowing up the centers of $F_k$, we get a tower of section-blow-ups $X_{k-1}\to X_{k-1}\to\dots\to X_{0}\to X_{-1}:=X_C$.
        \item The tower of section-blow-ups $X_{k-1}\to X_{k-1}\to\dots\to X_{0}\to X_C$ is along $\tilde{B}$ with multiplicity $(a_0+a_1-c,a_2,\dots,a_k)$.
        \item $t_i:=\operatorname{mult}_{h_i\subseteq\tilde{E_C}}(\mathbf{N})=k-1-i$ for $i=-2, 0,1,2,\dots,k-1$.
        \item $\operatorname{mult}_{h_i\subseteq\tilde{E_C}}(\mathbf{D})=\sum_{j=2+i}^k(a_j-c)$.
        \item $\operatorname{mult}_{p}(\mathbf{D})=0$ for all $0\leq i+j\leq k-1$ for any closed point $p\in Q_{ij}\setminus\tilde{E}_C$ such that $c_{X_i}(p)\not\in c_{X_i}(F_k)$ if $i<k-1$.
        \item For any prime divisor $E_{\infty}$ over $Y_k$ centered on $F_k$, there is an integer $i\in\{-2,0,1,2,\dots,k-1\}$ such that $p\in Q_{i(k-1-i)}$.
    \end{enumerate}

\end{lem}
\begin{proof}
    By Lemma \ref{induction_on_layer}, $\sum_{j=0}^k(1+j)\tilde{F}_j$ is anti-nef on $Y_k$ over $X$. Let $Y_{\tilde{C}}$ be the blow-up of $Y_k$ along $\tilde{C}$ with the exceptional divisor $E_{\tilde{C}}$. By Lemma \ref{bup}, for $1\leq i\leq k$, $F_i=G_{i-1}$ as prime divisors over $X$. By Lemma \ref{bupmult}, $\operatorname{mult}_{h_{-2}}(\mathbf{M})=0$ and $\operatorname{ord}_{G_i/X_{i-1}}(\mathbf{M})=1$ for $0\leq i\leq k-1$. (1) then follows from Lemma \ref{bup} and the potentially anti-nefness of $\mathbf{M}|_{X_{k-1}}$.
    
    Computing $\operatorname{ord}_{F_j}(B)$ with Lemma \ref{bupmult}, we have (2).

    For a birational model $U$ over $X$ that is isomorphic over general points of $C$, $\mathbf{M}|_U\cdot \tilde{C}=N\cdot \tilde{C}=-k-1$, which show that $\operatorname{mult}_{x\in C}\mathbf{N}=k+1$. By Lemma \ref{primedivisors}, (4) follows when $i=-2$. Let $s$ be a smooth curve on $E_C$ such that $F_n$ intersects $\tilde{s}$ and $x'=h_{-2}\cap s$. Then $\operatorname{mult}_{x'\in s}\mathbf{N}\geq n+1$. Potentially nefness of $\mathbf{N}|_{X_j}$ then shows that $2+t_0\leq t_{-2}=k+1$, $t_0\geq t_{-2}-2$, $t_{j}\geq t_{j-1}-1$, $2t_0\geq t_{-2}+t_1-2+1$, and $2t_j\geq t_{j-1}+t_{j+1}-1+1$ for $j\geq 1$. This shows (3).

    Considering $a_j=1$ for all $j$ for natural numbers $k$, (3) follows when $c=0$. (4) and (5) then follows since $\operatorname{ord}_{Q_{ij}/X}(\mathbf{N})=\operatorname{ord}_{Q_{ij}/X}(\mathbf{M})=2+i+j$ for $0\leq i+j\leq k-1$.

    (6) follows by (4) and (5), considering the contribution of $a_k$ in $\operatorname{ord}_{E_\infty/X}(\mathbf{D})$.
\end{proof}

\begin{lem}\label{d}
    Let $(X,dD)$ be an lc sub-pair of dimension $2$ such that $X$ and $D$ are smooth at $x\in X$, where $D$ is a prime divisor. 
    Suppose $\mathbf{M}$ is a b-nef$/X$ b-$\mathbb{R}$-Cartier divisor over $X$ with $\operatorname{mult}_{x\in D}\mathbf{M}\leq 1$, then $\operatorname{mld}_x(X,D+\mathbf{M})\geq 1-d$.
\end{lem}
\begin{proof}
We may assume that $x$ is a closed point.

We first assume that $d=1$.
Let $E$ be a divisor over $x$.
 Let $X_{k}\to\dots\to X_{1}\to X_{0}:=X$ be the sequence of blow-ups along the centers of $E$ such that $E$ is a divisor on $X_k$ but not on $X_j$ for $j<k$. Denote the morphism $X_{j}\to X_{(j-1)}$ by $\varphi_{j}$, the exceptional divisor by $E_{j}$ and the blow-up center by $x_{j-1}$.

Denote by $\tilde{D}$ the strict transform of $D$ on each $X_i$.
To show that $a_E(X,D+\mathbf{M})\geq 0$, we may assume that $x_l$ is not a smooth closed point of $\tilde{D}+\sum_{j=1}^lE_j$ for $0\leq l<k$, otherwise we may replace $(X,D)$ by $(X_l,E_l)$.
This means that the dual graph $G$ of $\tilde{D}+\sum_{j=1}^kE_j$ is a chain with leaves $\tilde{D}$ and $E_1$.
Express $G$ by the sequence $\tilde{D}:=F_0, F_{1}, F_{2}, \dots, F_{j}=E_k, \dots, F_{k}=E_1$.

Let $\mathbf{N}=[-E_1]$ be the b-nef$/X$ divisor.
Then $(X,D+\mathbf{N})$ is generalized lc.

Let $d_i=-F_{i}^2$, $n_i=\operatorname{ord}_{F_{i}/X}(\mathbf{N})$ and $m_i=\operatorname{ord}_{F_{i}/X}(\mathbf{M})$.
Then $d_i\geq 2$ for $0<i\neq j$, $n_0=m_0=0$ 
and $m_1\leq n_1=1$ by assumption.
For $1\leq i<j$ we have $F_i\mathbf{N}|_{X_k}=n_id_i-n_{i+1}-n_{i-1}=0$ and $F_i\mathbf{M}|_{X_i}=m_id_i-m_{i+1}-m_{i-1}\geq 0$.
For $1\leq i\leq j$, let $n'_i=n_i-n_{i-1}$ and
$m'_i=m_i-m_{i-1}$.
Then we have $m'_1\leq n'_1$.

Now we use induction to show that $m_i\leq n_i$ and $m'_i\leq n'_i$ for $1\leq i\leq j$.
Let $k\leq j-1$ be a natural number. Assume that $m_i\leq n_i$ and $m'_i\leq n'_i$ for $1\leq i\leq k$.
Then $m'_{k+1}=m_{k+1}-m_k\leq m_k-m_{k-1}+(d_k-2)m_k\leq n_k-n_{k-1}+(d_k-2)n_k=n_{k+1}-n_k=n'_{k+1}$.
On the other hand, $m_{k+1}=m'_{k+1}+m_k\leq n'_{k+1}+n_k=n_{k+1}$.
By induction on $k$, we have $m_i\leq n_i$ and $m'_i\leq n'_i$ for $1\leq i\leq j$.

In particular, we have $a_E(X,D+\mathbf{M})=a_{F_j}(X,D+\mathbf{M})=a_{F_j}(X,D)-\operatorname{ord}_{F_j}\mathbf{M}=
a_{F_j}(X,D)-m_j\geq a_{F_j}(X,D)-n_j=a_{F_j}(X,D)-\operatorname{ord}_{F_j}\mathbf{N}=a_{F_j}(X,D+\mathbf{N})\geq 0$.
This shows the lemma when $d=1$.

Now we consider the general case.
Let $E$ be a divisor over $x$. Then we have $a_E(X,dD+M)=a_E(X,D+M)+(1-d)\operatorname{ord}_ED\geq 0+(1-d)=1-d$.
\end{proof}

\begin{lem}\label{a<=1_original}
    Let $k$ be a natural number.
    Let $(X,dD)$ be a log smooth pair of dimension $3$, where $D$ is a prime divisor and $0\leq d\leq 1$. Let $f:X\to X_{0}$ be a birational morphism. Let $\mathbf{M}$ be a b-nef/$X_0$ b-divisor over $X$ such that $(X,dD+\mathbf{M})$ is log canonical as a generalized pair over $X_0$, and let $a\leq 1$ be a positive real number. Suppose that $C\subseteq D$ is a smooth exceptional$/X_0$ curve such that $C\cdot \mathbf{M}_X=(1-k)a$, $C\cdot_D C=1-k$, and $\operatorname{mult}_{C\subseteq D}\mathbf{M}=\operatorname{mult}_{C}\mathbf{M}=a$. Then $\operatorname{mld}_C(X,dD+\mathbf{M})\geq 1-d$.
\end{lem}
\begin{proof}
    Suppose $E$ is a prime divisor over $X$ with $a_E(X,dD+\mathbf{M})< 1-d$. First, suppose that $c_X(E)=C$. Cutting by a general hyperplane section of $X$, this contradicts Lemma \ref{d}.Therefore, we may assume that $c_X(E)=x\in C$ is a closed point on $X$.
    
    Let $b:=\operatorname{mult}_x\mathbf{M}$, which is at least $\operatorname{mult}_C\mathbf{M}=a$ by Lemma \ref{mult}.
    Let $X_1\to X$ be the blow-up of $X$ at $x$ with the exceptional divisor $E_1$. Let $l$ be the intersection of $E_1$ and $\tilde{D}$ on $X_1$ and let $C'$ be the strict transform of $C$ on $X_1$. Then $\operatorname{mult}_{C'\subseteq\tilde{D}}\mathbf{M} =\operatorname{mult}_{C'}\mathbf{M}=a$, $C'\cdot_{\tilde{D}}C'=-k$, and $C'\cdot\mathbf{M}_{X_1}=(1-k)a-b$.
    Let $X_{C'}$ be the blow-up of $X_1$ along $C'$ with the exceptional divisor $E_{C'}$ and let $C_1$ be the intersection of $E_{C'}$ and $\tilde{D}$ on $X_{C'}$. Then $\operatorname{mult}_{C_1}\mathbf{M}=0$ and therefore $C_1\cdot \mathbf{M}_{X_l}=(1-k)a-b+ka\geq 0$, which, together with $b\geq a$, implies $a=b$.

    We claim that $\operatorname{mult}_{l}\mathbf{M}=0$. Suppose $\operatorname{mult}_{l}\mathbf{M}>0$. Let $X_{l}$ be the blow-up of $X_1$ along $l$ and let $C_l$ be the strict transform of $C$ on $X_l$. Then $\operatorname{mult}_{C_l\subseteq\tilde{D}}\mathbf{M} =\operatorname{mult}_{C_l}\mathbf{M}=a$,$C_l\cdot_{\tilde{D}}C_l=C'\cdot_{\tilde{D}}C'=-k$, and $C_l\cdot\mathbf{M}_{X_1}<C'\cdot\mathbf{M}_{X_1}=-ka$. Let $X_{C_l}$ be the blow-up of $X_l$ along $C_l$ with the exceptional divisor $E_{C_l}$ and let $C_{l1}$ be the intersection of $E_{C_l}$ and $\tilde{D}$ on $X_{C_l}$. Then $\operatorname{mult}_{C_{l1}}\mathbf{M}=0$ and therefore $0\leq C_{l1}\cdot \mathbf{M}_{X_{C_l}}=C_l\cdot \mathbf{M}_{X_l}+ka<-ka+ka= 0$, which is a contradiction and the claim follows.

    Let $x_1:=l\cap C'$. Then $\operatorname{mult}_{x_1}\mathbf{M}\geq \operatorname{mult}_{C'}\mathbf{M}=a$. Since $\operatorname{mult}_{l}\mathbf{M}=0$ and $l\cdot \mathbf{M}_{X_1}=a$, we have $\operatorname{mult}_{x_1}\mathbf{M}=a$ and $\operatorname{mult}_{p}\mathbf{M}=0$ for every closed point $p\in l\setminus\{x_1\}$.

    If $c_{X_1}(E)=x_1$, we may replace $(X,dD)$ by $(X_1,d\tilde{D})$ and $k$ by $k+1$. Hence, we may assume $c_{X_1}(E)\neq x_1$. Then $c_{X_1}(E)\not\subseteq \tilde{D}$. If $\dim c_{X_1}(E)=0$, then $a_{E_1}(X,dD+\mathbf{M})=3-d-a\geq 1-d$ and by replacing $(X,dD)$ by $(X_1,0)$, we may assume $d=0$.

    So we may assume $c_{X_1}(E):= L$ is a curve. Then $\operatorname{mult}_{L\subseteq E_1}\mathbf{M}\leq a\leq 1$. Cutting by a general hyperplane section of $X_1$, by Lemma \ref{d}, we have $a_E(X,dD+\mathbf{M})\geq 1-\big(1-a_{E_1}(X,dD+\mathbf{M})\big)=3-d-a\geq 1-d$, which proves the lemma.
\end{proof}
\begin{lem}\label{a<=1}
    Let $X/Z$ be birational threefolds. 
    Let $x\in X$ be a closed point and let $D_1$ and $D_2$ be distinct prime exceptional divisors on $X$ over $Z$ containing $x$ such that $(X,D_1+D_2)$ is log smooth. Let $l=D_1\cap D_2$. Assume that $l$ is exceptional over $Z$. Suppose $\mathbf{M}$ is an exceptional b-nef$/Z$ b-divisor over $X$ such that $\operatorname{mult}_x{\mathbf{M}}:=a\leq 1$, $\operatorname{mult}_l{\mathbf{M}}=0$ and $l\cdot \mathbf{M}_X=a$. Then $\operatorname{mld}_x(X/Z,D_1+D_2+\mathbf{M})\geq 0$.
\end{lem}
\begin{proof}
    Suppose that $E$ is a prime divisor over $x\in X$ with $a_E(X/Z,D_1+D_2+\mathbf{M})<0$. Let $X_l$ be the blow-up of $X$ along $l$ with the exceptional divisor $E_l$ and let $l'$ be the fiber over $x$. Then, $\operatorname{mult}_{l'}\mathbf{M}=a$ by Lemma \ref{bupmult}.
    Moreover, by Lemma \ref{2.35}, $\operatorname{mult}_{l'\subseteq E_l}\mathbf{M}=a$ and $\operatorname{mult}_p\mathbf{M}=a$ for all $p\in l'$. By Lemma \ref{a<=1_original}, we may assume that $c_{X_l}(E)=l'\cap D_i:=x'$ for $i\in \{1,2\}$. Replacing $(x\in X,D_1+D_2)$ by $(x'\in X_l,D_i+E_l)$, we are done by induction.
\end{proof}

\begin{lem}\label{finitecurve}
    Fix natural numbers $N$ and $k$ and positive real numbers $\epsilon$, $n$ and $a$. Then there is a natural number $L$ depending only on $N,k,\epsilon,n$, and $a$ satisfying the following.

    Suppose $f:X\to X_0$ is a birational morphism of projective normal threefolds with exceptional prime divisors $D_1,D_2,\dots,D_k$, and $H_i\subseteq D_i$ are very ample divisors on $D_i$ with $H_i\cdot_{D_i} H_i\leq N$. Suppose $\mathbf{M}$ is an exceptional b-nef b-$\mathbb{Q}$-Cartier divisor over $X_0$ with $-\mathbf{M}_X=\sum_{i=1}^k d_iD_i$, where $d_i\in [0,n]$.
    Let $x\in D:=D_1$ be a closed point at which $X$ and $D$ are smooth. Suppose that for every curve $C\subseteq D$ containing $x$ along which $\mathbf{M}$ has positive multiplicity, $\epsilon\leq\operatorname{mult}_C\mathbf{M}\leq a-\varepsilon$.

    Suppose $Y_L\to Y_{L-1}\to\dots\to Y_0\to X$ is either a chain of point blow-ups or a tower of section-blow-ups at $x$ with the exceptional divisors $E_i\subseteq Y_i$ and the centers $y_i\in Y_i$ such that $\operatorname{mult}_{y_i}\mathbf{M}=\operatorname{mult}_{x}\mathbf{M}=a$ for every $0\leq i\leq L-1$. Then $y_{L-1}\cap \tilde{D}=\varnothing$.
\end{lem}
\begin{proof}
Let $K'$and $L'$ be fixed natural numbers to be determined later.
By Lemma \ref{boundedcurves}, the set of curves $C\subseteq D$ containing $x$ along which $\mathbf{M}$ has positive multiplicity form a bounded family. In particular, $\operatorname{mult}_xC\leq l$ for such curves $C$ for some fixed natural number $l$.
Let $G$ be a bounded relatively very ample exceptional divisor on $X$ over $X_0$ such that $C\cdot(\mathbf{M}_X+G)\geq L'$ for every curve $C\subseteq D$ containing $x$ along which $\mathbf{M}$ has positive multiplicity. Replacing $\mathbf{M}$ by $\mathbf{M}+[G]$, we may assume that $C\cdot\mathbf{M}_X\geq L'$ for every curve $C\subseteq D$ containing $x$ along which $\mathbf{M}$ has positive multiplicity.

Let $N'$ be a uniform bound number such that $D\cdot \mathbf{M}^2_X\leq N'$. Suppose $y_{L-1}\cap \tilde{D}\neq\varnothing$.

Then $\tilde{D}\cdot \mathbf{M}^2_{Y_i}=D\cdot \mathbf{M}^2_X-(i+1)a$ if $Y_i$ are point blow-ups.

Suppose that $Y_i$ are towers of section-blow-ups. Since $\mathbf{M}_{Y_i}$ are nef over $X$, we have 
\[
N'\geq D\cdot \mathbf{M}^2_X=D|_{Y_L}\cdot \mathbf{M}_{Y_L}^2\geq (\tilde{D}+E_0)\cdot \mathbf{M}_{Y_L}^2=\tilde{D}\cdot \mathbf{M}_{Y_L}^2+(L+1)a.
\]

Taking $L$ large enough, we may assume that $\mathbf{M}_{Y_L}$ is not nef over $X$.

So there is a curve $C\subseteq D$ containing $x$ along which $\mathbf{M}$ has positive multiplicity such that $\tilde{C}\cdot\mathbf{M}_{Y_L}<0$.
So, taking $L'\geq lK'$, we have $y_{K'}\cap\tilde{C}_0\neq \varnothing$ for some curve $C_0\subseteq D$ containing $x$ along which $\mathbf{M}$ has positive multiplicity.

For $1\leq i\leq K'$, we consider the following computations. By the boundedness of $C_0$ and $D$, on $Y_{i}$, $\tilde{C}_0\cdot \tilde{D}-\tilde{C}_0{\cdot_{\tilde{D}}} \tilde{C}_0\geq\alpha$, $\tilde{C}_0\cdot \mathbf{M}_{Y_{i}}=-ia+\alpha_1$, and $\tilde{C}_0{\cdot_{\tilde{D}}}\tilde{C}_0=-i+\alpha_2$ for some constants $\alpha,\alpha_1,\alpha_2\in (-b,b)$ for some fixed positive real numbers $b$ depending only on $N,k,\epsilon,n$ and $a$. Let $W$ be the blow-up of $Y_{K'}$ along $\tilde{C}_0$ with the exceptional divisor $E_{C_0}$ and let $C':=E_{C_0}\cap \tilde{D}$. Then $C'\cdot \mathbf{M}_W =\tilde{C}_0\cdot \mathbf{M}_{Y_{K'}}-(\operatorname{mult}_{\tilde{C}_0}\mathbf{M})C'\cdot E_{C_0}= - (a-\operatorname{mult}_{\tilde{C}_0}\mathbf{M}) K'+\alpha_1-(\operatorname{mult}_{\tilde{C}_0}\mathbf{M})\alpha_2\leq -\epsilon K'+\alpha_1+(a-\epsilon)|\alpha_2|$ and $C'\cdot_{E_{C_0}}C'=\tilde{C}_0\cdot_{Y_{K'}} \tilde{D}-\tilde{C}_0{\cdot_{\tilde{D}}} \tilde{C}_0\geq\alpha$, where $\tilde{D}$ is on $Y_{K'}$. Since $\operatorname{mult}_{C'\subseteq E_{C_0}}\mathbf{M}\leq a$, taking $K'$ large enough such that $\epsilon K'-b-(a-\epsilon)b>ab$, we get a contradiction to the nefness of $\mathbf{M}$ over $X_0$.
\end{proof}

\begin{lem}\label{curvebupinduction}
    Let $X$ be a smooth threefold, and let $C$ be a smooth curve on $X$ and $x\in C$ be a closed point. Let $n$ be a natural number. Suppose that $Y_n\to Y_{n-1}\to\dots\to Y_1\to Y_0:= X$ is a tower of point blow-ups at $x$ with the exceptional divisors $F_i\subseteq Y_i$ and centers $y_i\in Y_i$ with $y_i\in \tilde{C}$.
    
    Let $X_C$ be the blow-up of $X$ along $C$ with the exceptional divisor $E_C$. 

    Suppose that $E$ is a prime divisor over $X$ such that $c_{Y_n}(E_{\infty})\subseteq F_n$ and $c_{E_C}(E):=x_C$ is a closed point.

    Let $Z_n\to Z_{n-1}\to\dots\to Z_1\to Z_0:=X_C$ be the tower of blow-ups of centers of $E_{\infty}$ with the exceptional divisor $G_i\subseteq Z_i$ and the centers $z_i\subseteq Z_i$.

    Suppose $k\leq n$ is a non-negative integer such that $z_k$ is a closed point and either $\dim z_{k+1}>0$ or $k=n$.

 Then $z_{k}\in\tilde{E}_C$, and either $n=k$ or $z_{k+1}=G_{k+1}\cap \tilde{E}_C$.
\end{lem}
\begin{proof}
    Let $f$ be the fiber of $E_C\to C$ over $x$.
    Let $W_n\to W_{n-1}\to\dots\to W_1\to W_0:= X_C$ the tower of blow-ups along $f\subseteq E_C$ with the exceptional divisors $H_i\subseteq W_i$.
    By Lemma \ref{bup}, $W_n$ is the blow-up of $Y_n$ along $\tilde{C}$.

    By assumption, $z_0:=c_{X_C}(E_{\infty})$ is a closed point.

    Since a center of a prime divisor is always an irreducible subvariety, $c_{W_{n-1}}(E)=H_{n-1}\cap\tilde{E}_C\cap \varphi_{n-1}^{-1}(z_0)$ is a closed point, where $\varphi_{n-1}$ is the morphism $W_{n-1}\to X_C$ constructed above.

    It follows inductively on $k>0$ that if $T$ is a birational model over $X_C$ and $C'\subseteq \tilde{E}_C\subseteq T$ is a curve smooth at $c_T(E_{\infty})\in C'$ with $X_{C'}$ the blow-up of $T$ along $C'$ such that the tower of blow-ups of length $k-1$ of centers of $E_{\infty}$ starting form $X_{C'}$ has blow-up centers closed points in $\tilde{E}_C$ followed by possibly a curve as the intersection of the last exceptional divisor with $\tilde{E}_C$, then the tower of blow-ups of length $k$ of centers of $E_{\infty}$ starting form $T$ has blow-up centers closed points in $\tilde{E}_C$ followed by possibly a curve as the intersection of the last exceptional divisor with $\tilde{E}_C$, which shows the lemma.
\end{proof}

\begin{lem}
    Let $X\to Z$ be a birational morphism of smooth surfaces.
    Let $\delta$ be a positive real number and $N$ be a natural number.
    Let $M$ be a b-nef b-$\mathbb{R}$-Cartier divisor with coefficients at least $\delta$ over $Z$.
    Let $D$ be a reduced divisor on $X$ exceptional over $Z$ such that $D\cdot M_X\leq N$.
    Let $X_k\to X_{k-1}\to \dots \to X_1\to X_0:=X$ be any sequence of blow-ups at closed point centers $x_i\in X_i$.
    Let $D'_i\subseteq X_i$ be the reduced divisor supported on the pullback of $D$ on $X_i$.
    Then the number of $i\in [0,k-1]$ satisfying $\operatorname{mult}_{x_i}(D'_i)>1$ and $\operatorname{mult}_{x_i}(M)>0$ is at most $\frac{N}{\delta}$.
\end{lem}
\begin{proof}
     Since the statement is local over $D$, we may assume that $x_i\in D'_i$ for all $i$.
     Let $M_i$ be the image of $M$ on $X_i$.
     Denote the blow-up $X_i\to X_{i-1}$ by $\varphi_i$ and denote $E_i$ by the exceptional divisor.
     Then for $0\leq i\leq k-1$, we have
     $D'_{i+1}M_{i+1}=\big(\varphi^*_{i+1}D'_i-(\operatorname{mult}_{x_i}(D'_i)-1)E_{i+1}\big)\big(\varphi^*_{i+1}M_i-\operatorname{mult}_{x_i}(M)E_{i+1}\big)=D'_iM_i-(\operatorname{mult}_{x_i}(D'_i)-1)\operatorname{mult}_{x_i}(M)$.
    Since $(\operatorname{mult}_{x_i}(D'_i)-1)\in \mathbb{Z}_{\geq 0}$ and $\operatorname{mult}_{x_i}(M)\in \{0\}\cup [\delta,\infty)$,
    we have $(\operatorname{mult}_{x_i}(D'_i)-1)\operatorname{mult}_{x_i}(M)\in \{0\}\cup [\delta,\infty)$ and the statements follows by the assumption $D\cdot M_X\leq N$ and the nefness over $Z$ of $M_i$.
\end{proof}
\begin{lem}\label{cut}
Fix a bounded family of varieties $\mathcal{P}$, a natural number $n$ and a positive real number $\delta$.
Then there is a natural number $N$ satisfying the following.
    Let $Y\to Z$ be birational varieties of dimension $3$.
    Assume that $(Y,D)$ is  log smooth, where $D\in \mathcal{P}$ is an effective divisor on $Y$ exceptional over $Z$ onto a closed point of $Z$,
    $B$ is an effective divisor on $Z$ and $M'\subseteq Y'\to Y$ is an exceptional relatively nef over $Z$ $\mathbb{R}$-Cartier divisor such that $\operatorname{ord}_{D_0}(B+[M'])\leq n$ for every component $D_0$ of $D$.
    Assume that coefficients of $B$ and $M'$ are at least $\delta$.
    Let $Y_k\to Y_{k-1}\to \dots \to Y_1\to Y_0:=Y$ be any sequence of blow-ups at centers $y_i\in Y_i$.
    Let $D'_i\subseteq Y_i$ be the reduced divisor supported on the pullback of $D$ on $Y_i$.
    Let $M_i$ be the image of $[M']$ on $Y_i$.
    Then the number of $i\in [0,k-1]$ satisfying $\operatorname{mult}_{y_i}(D'_i)>1$, $\operatorname{mult}_{y_i}(\tilde{B}+M_i)>0$ and $\dim c_Y(y_i)>0$ is at most $N$.
\end{lem}
\begin{proof}
There is a fixed natural number $k$ such that for each component $D_j$ of $D$, there is a very ample divisor $H_{D_j}$ on $D_j$ with volume at most $k$.
Let $C_0=\sum C_j$, where $C_j$ is a general member of $|H_{D_j}|$ and the sum is taken over irreducible components $D_j$ of $D$.
Then there is a fixed natural number $L$ such $C_0(\tilde{B}+M_i)\leq L$.
Defined inductively $C_i\subseteq Y_i$ as follows.
If $\dim c_Y(y_i)=0$, then we set $C_{i+1}$ to be the strict transform of $C_i$.
If $\dim c_Y(y_i)>0$, then we set $C_{i+1}$ to be the strict transform of $C_i$ plus a general fiber of $E_{i+1}\to y_i$, where $E_{i+1}$ is the exceptional divisor of $Y_{i+1}\to Y_i$.
Then we have $0\leq C_{i+1}(\tilde{B}+M_{i+1})\leq C_{i}(\tilde{B}+M_{i})-(\operatorname{mult}_{y_i}(D'_i)-1)\operatorname{mult}_{y_i}(\tilde{B}+M_i)\dim c_Y(y_i)$.
\end{proof}
\begin{thm}[{\cite[Theorem 5.18]{HLS24}}]\label{HLS}
    Let $d$ be a positive integer, $\alpha$ a positive real number, and $\Gamma\subseteq [0,1]$ a closed DCC set. Then there exist a finite set $\Gamma'\subseteq \Gamma$ and a projection $g:\Gamma\to \Gamma'$ (i.e. $g\circ g=g$) depending only on $d$, $\alpha$, and $\Gamma$ satisfying the following.
    
    Let $(X,B:=\sum_{i=1}^sb_iB_i)$ be an lc pair of dimension $d$, such that each $b_i\in\Gamma$ and each $B_i$ is a $\mathbb{Q}$-Cartier Weil divisor. Then
    \begin{itemize}
        \item
        $\gamma+\alpha\geq g(\gamma)\geq \gamma$ for any $\gamma\in \Gamma$, $g(\gamma')\geq g(\gamma)$ for any $\gamma'\geq\gamma$ in $\Gamma$
        \item $\big(X,\sum_{i=1}^sg(b_i)B_i\big)$ is lc.
    \end{itemize}
\end{thm}

\begin{thm}\label{main2}
Fix natural numbers $n$ and $N$, a positive real number $\delta$ and a DCC set $I\subseteq\mathbb{R}$. Then there exists an ACC set $J\subseteq\mathbb{R}_{\geq 0}$ satisfying the following.
If $(X,\Delta+B+\mathbf{M})$ is a generalized sub-lc pair of dimension $2$ satisfying the following conditions
\begin{itemize}
  \item $(X+\Delta)$ is a log smooth pair,
    \item $B=B^+-B^-$ with $B^+$ and $B^-$ effective divisors,
    \item $\mathbf{M}$ is a b-nef b-divisor over $X$ with image $M$ on $X$,
    \item the coefficients of $\Delta$, $B$ and $\mathbf{M}$ are in $I$,
    \item the coefficients of $B^+$ and $\mathbf{M}$ are at least $\delta$,
\end{itemize}
then for any closed point $x\in X$ with $\big(\red{B^-}\cdot (B^++M)\big)_x\leq N\delta$,
we have $\operatorname{mld}_x(X,\Delta+B+\mathbf{M})\in J$.
\end{thm}
\begin{proof}
Assume $x\in X$ is a closed point.
Let $\phi_x:X_x\to X$ be the blow-up of $X$ at $x$ with exceptional divisor $E_x$.
Let $\tilde{B^+}$,
$\tilde{B^-}$ and $\tilde{\Delta}$ be the strict transforms of $B^+$, $B^-$ and $\Delta$, respectively.
Let $M_x=\mathbf{M}|_{X_x}$.
Let $a=a_{E_x} (X,B+\mathbf{M})$.
 If $(\operatorname{mult}_xB^+
 +\operatorname{mult}_x\mathbf{M})=0$, then $E_x$ computes $\operatorname{mld}_x(X,\Delta+B+\mathbf{M})$.
Since $a=2-\operatorname{mult}_x(\Delta+B)-\operatorname{mult}_x\mathbf{M}$ is contained in the ACC set $ 2-I$, we may assume that 
$(\operatorname{mult}_xB^+
 +\operatorname{mult}_x\mathbf{M})
 \geq\delta$.





Assume that the statement does not hold. Then there is a sequence of generalized sub-lc pairs $(X_i,B_i+\mathbf{M}_i)$, for $i=1,2,3,\dots$ with $\Delta_i$, $\mathbf{M}_i$ satisfying the conditions in the statement such that $\{m_i\}_i$ is an increasing sequence for some
closed points $x_i\in X_i$, and positive real numbers $m_i= \operatorname{mld}_{x_i}(X,B_i+\mathbf{M}_i)$.

Let $F_i$ be a divisor over $X_i$ that computes $\operatorname{mld}_{x_i}(X,B_i+\mathbf{M}_i)$.

Let $X_{ij}\to\dots\to X_{i1}\to X_{i0}:=X_i$ be the sequence of blow-ups along the centers of $F_i$. Denote the morphism $X_{ij}\to X_{i(j-1)}$ by $\varphi_{ij}$, the exceptional divisor by $F_{ij}$ and the blow-up center by $x_{i(j-1)}$.
Let $M_{ij}=\mathbf{M}|_{X_{ij}}$.

By Lemma \ref{lim}, we may assume that $F_{ii}\neq F_i$ for all $i$.

    Possibly replacing $F_i$ by some $F_{ij}$, we may assume that $a_{F_{ij}}(X_i,B_i+\mathbf{M}_i)>m_i$ for all $F_{ij}\neq F_i$.
    This implies that $\operatorname{mult}_{x_{ij}}(B^+_{ij}+\mathbf{M}_i)\geq \delta$ for all $F_{ij}\neq F_i$, where $B^+_{ij}$ is the strict transform of $B^+$ on $X_{ij}$.

We claim that there is a fixed natural number $n$ such that $\operatorname{ord}_{F_i}\red{B^-_i}\leq n$ for all $i$.
By Theorem \ref{surf}, it is enough to show that there is a fix natural number $k$ such that $x_{ik}\not\in B^-_{ik}$ for all $i$, where $B^-_{ik}$ is the strict transform of $B^-$ on $X_{ik}$.
By Lemma \ref{lim}, we may assume, by contrast, that $x_{ii}\in B^-_{ii}$.

Then, for $j<i$, we have
\begin{align*}
    &\red{B^-_{ij}}\cdot \big({B^+_{ij}}+{M_{ij}}\big)  \\
    =&\big(\red{B^-_{i(j+1)}}+\operatorname{mult}_{x_{ij}}\red{B^-_{ij}}F_{i(j+1)}\big)\cdot
    \big(\red{B^+_{i(j+1)}}+\operatorname{mult}_{x_{ij}}{B^+_{ij}}F_{i(j+1)}\big)    \\
    +&\big(\red{B^-_{i(j+1)}}+\operatorname{mult}_{x_{ij}}\red{B^-_{ij}}F_{i(j+1)}\big)\cdot
    \big(M_{i(j+1)}+\operatorname{mult}_{x_{ij}}\mathbf{M}_iF_{i(j+1)}\big)\\
    =&\red{B^-_{i(j+1)}}\cdot ({B^+_{i(j+1)}}+{M_{i(j+1)}})+\operatorname{mult}_{x_{ij}}\red{B^-_{ij}}\big(\operatorname{mult}_{x_{ij}}{B^+_{ij}}+\operatorname{mult}_{x_{ij}}\mathbf{M}_i\big)\\
    \geq &\red{B^-_{i(j+1)}}\cdot ({B^+_{i(j+1)}}+{M_{i(j+1)}})+\delta.
\end{align*}
Therefore $0\leq\red{B^-_{ii}}\cdot \big({B^+_{ii}}+{M_{ii}}\big)\leq N-i\delta$, which is impossible for $i$ large enough, and so the claim holds.

Since $\operatorname{ord}_{F_i}\red{B^-_i}\leq n$ for all $i$, replacing $X_i$ by some $X_{ik}$ for some $k$, we may assume that $\operatorname{ord}_{F_{ij}}\red{B^-_i}\leq 1$ and that $x_{i(j+1)}$ is not in the strict transform of $F_{ij}$ for all $j\leq i$.
In particular, for all $j\leq i$, we have $\operatorname{ord}_{F_{ii}}{B^-_i}=b_i$ for some non-positive real number $b_i\in I\cup\{0\}$.

Let $a_{ij}=\operatorname{mult}_{x_{ij}}(B^+_{ij}+\mathbf{M})$. Then for each $i$, we have a non-increasing sequence $\{a_{ij}\}_j$ and $a_{ij}\in \sum I_{\geq 0}$
, where $I_{\geq 0}=I\cup\mathbb{R}_{\geq 0}$.
Since $a_{F_{ij}}(X,B_i+\mathbf{M}_i)=1+j-b_i-\sum_{k=1}^ja_{ik}\geq 0$,
we have $\sum_{k=1}^j(a_{ik}-1)\leq\min I+1 $.
By Lemma \ref{lim}, we may assume that $a_{i1}\leq 1$.

By Lemma \ref{B=0}, there is a b-nef$/X_i$ b-$\mathbb{R}$-Cartier divisor $\mathbf{B}^+_i$ over $X_i$ such that
$\operatorname{ord}_{F_{ij}/X_i}(\mathbf{B}^+_i)=\operatorname{ord}_{F_{ij}}B^+_i$ for all $i, j$.
Since $\operatorname{mult}_{x_i}(\mathbf{B}^+_i+\mathbf{M}_i)=a_{i0}\leq 1$, by Lemma \ref{a},
we have $\operatorname{mult}_{x_{i1}\subseteq F_{i1}}(\mathbf{B}^+_i+\mathbf{M}_i)\leq a_{i0}\leq 1$.
By Lemma \ref{d}, $m_i=a_{F_i}(X_i,B^-_i+\mathbf{B}^+_i+\mathbf{M}_i )\geq a_{F_{i1}}(X_i,B^-_i+\mathbf{B}^+_i+\mathbf{M}_i ) $, which gives a contradiction.
\end{proof}

\begin{cor}
   Fix a bounded family of (not necessarily smooth) varieties $\mathcal{P}$. Then, Theorem \ref{main2} hold for non-smooth varieties $X$ in $\mathcal{P}$.
\end{cor}

\section{ACC for Bounded Generalized Sub-Pairs in Arbitrary Dimension}\label{section:main}
In this section, we prove that the ACC holds for bounded families of generalized sub-pairs in any dimension.

\begin{thm}\label{sub}
Fix natural numbers $d$ and $n$, positive real number $\epsilon$ and $\delta$, and a DCC set $I\subseteq \mathbb{R}_{\geq 0}$. Then there exists a positive real number $t$ satisfying the following.
If $(X,B+\mathbf{M})$ is a generalized sub-$\epsilon$-lc sub-pair and $D+\mathbf{N}$ is a generalized boundary of it such that
\begin{itemize}
\item $\dim X = d$.
\item we may write $B=E-F$ for some effective divisors $E$ and $F$,
\item there is a very ample divisor $H$ on $X$ with $H^d\leq n$ and $H-(E+F+D+\mathbf{M}|_X+\mathbf{N}|_X)\geq 0$ and
\item the coefficients of $E$, $F$, $D$, $\mathbf{M}$ and $\mathbf{N}$ are at least $\delta$.
\end{itemize}
then $\big(X,(B+tD)+(\mathbf{M}+t\mathbf{N})\big)$ is generalized sub-lc.
\end{thm}
\begin{proof}
By \cite[Theorem 3.14]{birkar21}, there is a bounded family of log smooth couples $\mathcal{P}$ such that for every generalized sub-pair $(X,B+\mathbf{M})$ and every generalized boundary $D+\mathbf{N}$ as in the assumption of the theorem, there is an element $(\overline{X},\overline{\Sigma})\in\mathcal{P}$ satisfying that
\begin{itemize}
\item there is a birational morphism $\varphi:\overline{X}\to X$,
\item $\overline{\Sigma}$ contains the exceptional divisors of $\varphi$ and the support of the birational transforms of $E$, $F$, and $D$, and
\item $\mathbf{M}$ and $\mathbf{N}$ descend to $\overline{X}$.
\end{itemize}
Furthermore, the existence of $H$ guarantees that there exists a natural number $k$ depending only on $d$ and $n$ such that the coefficients of $\varphi^*(D+\mathbf{N})$ are at most $k$.
The theorem follows by taking $t=\frac{\epsilon}{k}$.
\end{proof}

\begin{thm}(ACC for bounded sub-lct)
Fix natural numbers $d$ and $n$ and fix a DCC set $I\subseteq \mathbb{R}_{\geq 0}$. Let $P(d,n,I)$ be the set of generalized sub-pairs $(X,B+\mathbf{M})$ such that 
\begin{itemize}
\item $\dim X = d$,
\item $X$ is $\mathbb{Q}$-factorial,
\item we may write $B=E-F$ for some some effective divisors $E$ and $F$,
\item there is a very ample divisor $H$ on $X$ with $H^d\leq n$ and $H-(E+F+\mathbf{M}|_X)\geq 0$,
\item $F$ is integral, and
\item $\operatorname{coeff}E\in I$ and $\operatorname{coeff}M\in I$.
\end{itemize}
Then the set
\[\{ \operatorname{mld}(X,B+\mathbf{M})|(X,B+\mathbf{M})\in P(d,n,I)\}\]
satisfies ACC.  
\end{thm}
\begin{proof}
Assume that there exists a sequence of generalized pairs $(X_i,B_i+\mathbf{M}_i)$ in $P(d,n,I)$ such that $\{m_i:= \operatorname{mld}(X_i,B_i+\mathbf{M}_i)\}_i$ is a strictly increasing sequence.
Passing to a subsequence, we may assume that there exists some positive number $\epsilon$ such that $m_i\geq \epsilon$ for all $i$.
Write $B_i=E_i-F_i$ for some effective divisors $E_i$ and $F_i$ such that the coefficients of $E_i\in I$ and $F_i$ is integral. 
By Theorem \ref{sub}, there exists a positive number $\delta$ such that $(X_i,(1+\delta) E_i-F_i+(1+\delta) \mathbf{M}_i)$ is generalized log canonical for all $i$. 

Write $E_i=\sum_{j=1}^{s_i} e_{i,j}E_{i,j}$ and $M'_i=\sum_{j=1}^{t_i} a_{i,j}M'_{i,j}$ for some $ e_{i,j}\in I$, $ a_{i,j}\in I$, $E_{i,j}$ integral divisors on $X_i$ and $M'_{i,j}$ Cartier divisors on $X'_i$, on which $\mathbf{M}_i$ descends to $M'_i$.
Since $\{(X_i,E_i+\mathbf{M}_i)\}_i$ is a bounded family of generalized pairs, we may assume that $s_i=t_i=s$ for some $s\in \mathbb{N}$ for every $i$. Furthermore, $e_{i,j}$ and $a_{i,j}$ are bounded above for all $i$ and $j$.
Passing through a subsequence, we may assume that sequences $\{e_{i,j}\}_i$ and  $\{a_{i,j}\}_i$ are non-decreasing for all $1\leq j\leq s$.
Let $e_{0,j}=\lim_{i\to 0}e_{i.j}$ and $a_{0,j}=\lim_{i\to 0}a_{i.j}$.
Consider the generalized pairs $(X_i,\hat{B_i}+\hat{\mathbf{M}}_i)$ where $\hat{B_i}=(\sum_{j=1}^se_{0,j}E_{i,j})-F_i$ and $\hat{\mathbf{M}}_i=\sum_{j=1}^{s}a_{0,j}[M'_{i,j}]$.
By construction, we have that $B_i\leq \hat{B_i}\leq (1+c_i\delta)E_i-F_i$ and $M'_i\leq \hat{\mathbf{M}}_i|_{X'_i}\leq (1+c_i\delta) M_i$ for some positive numbers $c_i$'s such that $\lim_{i\to\infty} c_i=0$.
Therefore, $\operatorname{mld(X_i, \hat{B_i}+ \hat{\mathbf{M}}_i)}\in [(1-c_i)m_i,m_i]$.
Passing to a subsequence, we may assume that $\{\operatorname{mld}(X_i, \hat{B_i}+ \hat{\mathbf{M}}_i)\}_i$ is strictly increasing.

On the other hand, by Theorem \ref{discrete}, $\{\operatorname{mld}(X_i, \hat{B_i}+ \hat{\mathbf{M}}_i)\}_i$ is a discrete set. This gives the desired contradiction, and the proof is completed.
\end{proof}

\section{ACC for Bounded 3-Folds with Arbitrary Boundaries}
In this section, we prove that the ACC holds for bounded families of threefolds.

\begin{thm}[ACC for mlds for bounded threefolds]\label{original}
Fix a DCC set $I$ of positive real numbers and a natural number $N$.
Then there is an ACC set $J=J(I,N)\subseteq\mathbb{R}_{\geq 0}$ satisfying the following. 

If $(X,D)$ is a lc pair of dimension $3$, such that
\begin{itemize}
    \item $\operatorname{coeff} D\in I$, and
    \item there is a very ample divisor $H$ on $X$ with $H^3\leq N$,
\end{itemize}
then $\operatorname{mld}(X,D)\in J$.
\end{thm}

\begin{lem}\label{original_I0}
      Fix a DCC set $I$ of positive real numbers and a natural number $N$. If Theorem \ref{original} does not hold, then there is a natural number $N'$, a finite set $I_0\subseteq \mathbb{R}_{\geq 0}$, real numbers $0\leq m<m'$, and a sequence of $3$-dimensional lc pairs $(X_i,D'_i)$, and prime divisors $F_i$ over $X_{i}$, $i=1,2,3, \dots$, satisfying the following.

\begin{itemize}
\item $X_i$ is terminal,
\item $\operatorname{coeff} D'_i\in I_0$, 

\item there is a very ample divisor $H_i\subseteq X_i$ with $H_i^3\leq N'$,

\item $\operatorname{mld}(X_i,D'_i)=a_{F_i}(D'_i)=m$
, and 
\item $\operatorname{mld}(X_i,\frac{i-1}{i}D'_i)\geq m'$.
\end{itemize}
\end{lem}

\begin{proof} 
   Replacing $I$ by its closure $\overline{I}$, we may assume $I$ is closed.
   Suppose that Theorem \ref{original} does not hold. Then, for $i=1,2,3\dots$, there is a sequence of lc pairs $(X_i,D_i)$ and very ample divisors $H_i\subseteq X_i$ satisfying the assumptions of Theorem \ref{original} with prime divisors $F_i$ over $X_{i}$ such that $m_i:=\operatorname{mld}(X_i,D_i)=a_{F_i}(D_i)$ is strictly increasing.
   Replacing $X_i$ by their terminal models, we may assume that $X_i$ is terminal.
   Since $\{X_i\}$ is bounded, the $\operatorname{ord}_{c_{X_i}F_i}D_i$ is bounded. Dropping the components outside around the generic point of $c_{X_i}F_i$, we may write $D_i=\sum_{j=1}^k d_{ij}D_{ij}$ for some natural number $k$, $d_{ij}\in I$ and prime divisors $D_{ij}$.

   Passing through a subsequence, we may assume that for every $j$, $\{d_{ij}\}_{i=1}^{\infty}$ is a non-decreasing sequence with limit $d_{0j}$.

    Let $\alpha$ be a positive real number such that $(d_{0j},d_{0j}+\alpha)\cap I=\varnothing$, which is guaranteed by the DCC of $I$. 
   
   Let $I_0\subseteq I$ be a finite set and $g:I\to I_0$ be a projection given by Theorem \ref{HLS} with respect to $\alpha$. Then $\big(X_i,g(D_i)\big)$ is lc for all $i$. By the choice of $\alpha$, passing through a subsequence, we have $g(d_{ij})=d_{0j}$ and $\frac{i-1}{i}d_{0j}\leq d_{ij}$ for all $i,j$.
    Therefore, $\frac{i-1}{i}g(D_i)\leq D_i$ for all $i$.
    
   On the other hand, by Theorem \ref{discrete}, $\{\operatorname{mld}\big(X_i,g(D_i)\big)\}$ is a finite set. Therefore, passing through a subsequence, we have $m=\operatorname{mld}\big(X_i,g(D_i)\big)\leq m_i$ for some non-negative real number $m\leq m_i$ for all $i$. 
   
   Since $\frac{i-1}{i}g(D_i)\leq D_i$, for all $i\geq 2$, we have $\operatorname{mld}\big(X_i,\frac{i-1}{i}g(D_i)\big)\geq \operatorname{mld}\big(X_i,D_i\big)=m_i\geq m_2>m_1\geq m$. Dropping the first term of the sequence and taking $m'=m_1$ and $D'_i=g(D_i)$, the lemma follows.
   
\end{proof}

\begin{lem}\label{I0}
      Fix a DCC set $I$ of real numbers and a natural number $N$. If Theorem \ref{original} does not hold, then there is a natural number $N'$, a finite set $I_0\subseteq \mathbb{R}$, real numbers $0\leq m<m'$, and a sequence of birational morphisms $\phi_i:X_{i1}\to X_{i0}:=X_{i}$ of threefolds and prime divisors $F_i$ over $X_{i1}$, $i=1,2,3, \dots$, satisfying the following.

\begin{itemize}
\item $\phi_i$ is isomorphic over outside a subvariety of dimension $0$ of $X_i$,
\item $\Delta_i=-\Delta^-_i+\Delta^+_i$ is exceptional over $X_{i0}$, $\Delta^-_i\geq 0$, $\Delta^+_i\geq 0$,
\item $(X_{i1}:=X_i,\Delta_i)$ is log smooth,
\item $B_i\geq 0$ is a divisor on $X_{i1}$ which contains no exceptional component over $X_{i0}$,
    \item $\operatorname{coeff}\phi_i^*{\phi_i}_*B_i\leq N' $, and
    \item there is a very ample divisor $H_i\subseteq X_{i1}$ such that $M_i:=\phi_i^*{\phi_i}_*(H_i)-H_i$ satisfies $-M_i^3\leq N'$,
\item $\operatorname{coeff} \Delta_i\in I_0$, $\operatorname{coeff} B_i\in I_0$, 
\item $\operatorname{mld}(X_i^\circ,\Delta_i+B_i)=a_{F_i}(\Delta_i+B_i)=m$
, and $\operatorname{mld}\big(X_i^\circ,-\Delta_i^-+\frac{i-1}{i}(\Delta^+_i+B_i)\big)\geq m'$ for some neighborhood $X_i^\circ$ of $c_{X_{i1}}(F_i)$ in $X_{i1}$.
\end{itemize}
\end{lem}
\begin{proof}
    Suppose Theorem \ref{original} does not hold. Let $I_0$, $m$, $m'$, $(X_i,D_i)$ and $F_i$ be given as in Lemma \ref{original_I0}. Let $\{\phi_i:X_{i1}\to X_{i0}:=X_{i}\}$ be a bounded family of resolutions of $\{X_{i}\}$. 
    Since $X_{i}$ is terminal, we may assume $\phi_i$ is isomorphic over outside some finitely many closed points of $X_{i0}$.
    
    Write $\phi_i^*(K_{X_{i0}}+D_i)=K_{X_{i1}}+\Delta_i+B_i$, where $\Delta_i$ is exceptional and $B_i$ is the strict transform of $D_i$. Possibly extending $I_0$, we may assume $\operatorname{coeff} \Delta_i\in I_0$ and $\operatorname{coeff} B_i\in I_0$. Also, $\operatorname{coeff}\phi_i^*{\phi_i}_*B_i$ is bounded. Take $H_i\subseteq X_{i1}$ a bounded family of very ample divisor and set $M_i:=\phi_i^*{\phi_i}_*(H_i)-H_i$.
    Then $M_i$ and $\operatorname{Exc}(\phi_i)$ are bounded.
    Finally, let $\Delta^-_i:=N'\operatorname{Exc}(\phi_i)$ for some fixed natural number such that $\Delta^-_i+\phi_i^*(K_{X_i})-K_{X_{i1}}\geq 0$ and set $\Delta_i^+:=\Delta_i+\Delta_i^-$, and we are done.
\end{proof}
\begin{defi}
        For each non-negative integer $n$ and for each $(n+1)$-tuple $\mathbf{a}=(a_1,a_2,\dots,a_n)$ of real numbers, we define $\sum\mathbf{a}$ to be the value $\sum_{i=0}^na_i-n-2$.
\end{defi}

\begin{defi}\label{P}
  Let $I_0\subseteq \mathbb{R}_{\geq 0}$ be a finite set, $0\leq m$ and $c>0$ be real numbers, $n$ be a non-negative integer, and $\mathbf{a}=(a_0,a_1,\dots,a_n)$ be a $(n+1)$-tuple of non-negative real numbers. Let $\mathbf{o}\in \mathbb{C}^3$ be the origin, $E_\mathbf{o}$ be the exceptional divisor of the blow up of $\mathbb{C}^3$ at the $\mathbf{o}$, and let $\mathbf{M}=[-E_\mathbf{o}]$. We say that property $\mathrm{P}(I_0, m,\mathbf{a},c)$ holds if there are real numbers $m'>m$ and $t$ such that for every natural number $i$, there are a pair $(\mathbb{C}^3,S_i)$ and a smooth prime divisor $\Delta_i\subseteq \mathbb{C}^3$ satisfying the following conditions.
  \begin{enumerate}
      \item $\operatorname{coeff} S_i\in I_0$.
      \item There is a chain $Z_{ii}\to Z_{i(i-1)}\to\dots\to Z_{i1}\to Z_{i0}:=\mathbb{C}^3$ of towers of section-blow-ups $Z_{ik}:=Z_{ikn}\to Z_{ik(n-1)}\to\dots\to Z_{ik1}\to Z_{ik0}\to Z_{i(k-1)n}:=Z_{i(k-1)}$ with multiplicity $\mathbf{a}$ along $\tilde{S}_i$ over $\mathbf{o}\in \mathbb{C}^3$ and layers $F_{ik0}, F_{ik1}, \dots  F_{ikn}$ of $Z_{ik}\to Z_{i(k-1)}$ such that $c_{Z_{ik}}F_{i(k+1)0}=:x_{ik}\not\in\tilde{F}_{ik(n-1)}$
      for $1\leq k\leq i$.
      \item $\operatorname{ord}_{F_{ikl}}(\Delta_i)=1$ for all $k,l$.
      \item $\operatorname{mld}(\mathbb{C}^3,S_i+t\Delta_i)=m$ with $\mathbf{o}$ being a mld center.
      \item $\operatorname{mld}(\mathbb{C}^3,\frac{i-1}{i}S_i+t\Delta_i)\geq m'$.
      \item For $1\leq k\leq i$, either there is a non-lc center of $\big(Z_{ik},\tilde{S_i}+\Delta_{ik}+(1-\sum\mathbf{a})\mathbf{M}\big)/\mathbb{C}^3$ on $F_{ikn}
      $ or there is no non-lc center, not identical with $x_{ikn}$, of $\big(Z_{ik},\tilde{S_i}+\Delta_{ik}+(1-\sum\mathbf{a})\mathbf{M}/\mathbb{C}^3\big)$ on $\tilde{F}_{ikl}$ for all $0\leq l\leq n$, where $\Delta_{ik}$ is the exceptional divisor over $\mathbb{C}^3$ such that $\tilde{S_i}+\Delta_{ik}+(1-\sum\mathbf{a})\mathbf{M}$ is the crepant pullback of $S_i+(1-\sum\mathbf{a})\mathbf{M}$ on $Z_{ik}$. Note that $\mathbf{M}$ descends to its image on $Z_{ik}$.
      \item $\tilde{S}_i\cdot F_{i1n}=\sum_j d_j s_j$ for some positive real numbers $d_j$ and some lines $s_j\subseteq F_{i1n}$ with $c=\max_j\{d_j\}$.
  \end{enumerate}
\end{defi}~
\begin{lem}\label{step8}
Let $I_0\subseteq \mathbb{R}_{\geq 0}$ be a finite set, $m\geq 0$ and $c>0$ be real numbers, $n$ be a non-negative integer, and $\mathbf{a}=(a_0>1,a_1,\dots,a_n)$ be an $(n+1)$-tuple of non-negative real numbers. Suppose that $\mathrm{P}(I_0, m,\mathbf{a},c)$ holds. Then at least one of the following holds.
\begin{enumerate}
    \item $c\leq 1$.
    \item $c<a_1$, $n\geq 1$ and $\mathrm{P}(I_0, m,\mathbf{a}',c')$ holds, where $\mathbf{a}'=(a_0':=a_0+a_1-c,a_2,\dots,a_n,c)$ and $c'\leq c$ is some non-negative real number.
    \item $c=a_1$, $n\geq 1$ and $\mathrm{P}(I_0, m,\mathbf{a}',c')$ holds, where $\mathbf{a}'=(a_0,a_1,a_2,\dots,a_n,c)$ and $c'\leq c$ is some non-negative real number.
    \item $n=0$ and $\mathrm{P}(I_0, m,\mathbf{a}',c')$ holds, where $\mathbf{a}'=(a_0,c)$ and $c'\leq c$ is some non-negative real number.
    \item $\mathrm{P}(I_0, m,\mathbf{a},c')$ holds, where $c'< c$ is some non-negative real number.
\end{enumerate}
\end{lem}
\begin{proof}  
     Suppose $\mathrm{P}(I_0, m,\mathbf{a},c)$ holds with $c>1$. By Lemma \ref{sectionmult1} and Lemma \ref{a<=1}, if there is a non-lc center of $(Z_{ik}/\mathbb{C}^3,\tilde{S_i}+\Delta_{ik}+(1-\sum\mathbf{a})\mathbf{M})$ on $F_{ikn}$ for some $1\leq k\leq i$, the non-lc center must be contained in a line $s\subseteq F_{ikn}$ with $\operatorname{mult}_s\tilde{S}_i=c>1$. By Lemma \ref{sectionmult} and Lemma \ref{a<=1}, (6) in Definition \ref{P} is passed to the new chains of towers of section-blow-ups, constructed below.
     
     We consider the following cases separately. 
     
     Case (A): either $n\geq 1$ and $c=a_1$ or $n=0$. By Lemma \ref{sectionmult1}, for each $i$, there is a natural number $j=j(i)\leq i$ such that there is a chain $Z'_{ij}\to Z'_{i(j-1)}\to\dots\to Z'_{i1}\to Z'_{i0}:=\mathbb{C}^3$ of towers of section-blow-ups $Z'_{ik}:=Z'_{ik(n+1)}\to Z'_{ikn}\to\dots\to Z'_{ik1}\to Z'_{ik0}\to Z'_{i(k-1)(n+1)}:=Z'_{i(k-1)}$ with multiplicity $\mathbf{a}':=(a_0,a_1,\dots,a_n,c)=(a_0,c,\dots,c)$ along $\tilde{S}_i$ over $\mathbf{o}\in \mathbb{C}^3$ and layers $F'_{ik0}, F'_{ik1}, \dots  F'_{ik(n+1)}$ of $Z'_{ik}\to Z'_{i(k-1)}$ such that $c_{Z'_{ik}}F'_{i(k+1)0}:=x'_{ik}\not\in\tilde{F}'_{ik(n-1)}$
    for $1\leq k\leq j$ and $\tilde{S_i}\cdot F'_{ij(n+1)}=\sum d'_js'_j$ for some positive real numbers $d'_j<c$ and some distinct lines $s'_j$ on $F'_{ij(n+1)}$.
    By Lemma \ref{lim}, if $n\geq 1$ and $c=a_1$, either (3) holds or (5) holds.
    Similarly, by Lemma \ref{lim}, if $n=0$, either (4) holds or (5) holds.

    Case (B): $n\geq 1$, and $c<a_1$. By Lemma \ref{sectionmult1} and Lemma \ref{sectionmult}, (2) holds.
\end{proof}
\begin{lem}\label{step8_2}
     Under the notations of Lemma \ref{step8}, suppose that $\mathrm{P}(I_0, m,\mathbf{a},c)$ holds. Then, there is a non-negative integer $n'$, an $(n'+1)$-tuple of non-negative real numbers $\mathbf{a}'=(a_0'>1,a_1',\dots,a_{n'}')$ and a positive real number $c'\leq 1$ such that $\mathrm{P}(I_0, m,\mathbf{a}',c')$ holds.
\end{lem}
\begin{proof} 
    If $c\leq 1$, there is nothing to show. So we may assume that $c>1$.
    
    Replacing $I_0$ by $I_0\cup\{1\}$, we may assume that $1\in I_0$. Let $\epsilon$ be the minimum element of $(\sum I_0\cap[0,2]-\sum I_0\cap[0,2])\cap\mathbb{R}_{>0}$, which exists since $\sum I_0$ is discrete. 

    Note that $\sum\mathbf{a}\leq 0$ because $\operatorname{mld}(\mathbb{C}^3,S_i+t\Delta_i)=m$ for all $i$.

    We may observe that each condition of Lemma \ref{step8} either $\sum \mathbf{a'}-\sum\mathbf{a}\geq \epsilon$, $a_0'-a_0\geq \epsilon$, or $c-c'\geq \epsilon$.
    The lemma then follows from Lemma \ref{step8}, by induction on $(\sum\mathbf{a},a_0,2-c)$ under the lexicographical order. 
\end{proof}

\begin{thm}
   Let $I$ be a DCC set of real numbers and $N$ be a natural number.

   Suppose that $I_0$, $\phi_i:X_{i1}\to X_{i0}$, $\Delta_{i}$, and $F_i$, are given as in Lemma \ref{I0} for $i=1,2,3, \dots$.

    Inductively, for $j\geq 1$, set $x_{ij}=c_{X_{ij}}(F_{i})$, $\phi_{i(j+1)}:X_{i(j+1)}\to X_{ij}$ be the blow-up at $x_{ij}$ with exceptional divisor $E_{i(j+1)}$, and $a_{ij}=\operatorname{mult}_{x_{ij}}\tilde{B}_i\in \sum I_0^+\cap [0,-3\min\{I_0\}+3]$ where $I_{0}^{+}:=I_{0}\cap[0,\infty)$.

   Then, possibly replacing $I$ and $N$ and passing to a subsequence, we may assume that the followings hold.
   
   \begin{enumerate}
       \item $x_{ii}$ is a closed point for every $i$.
       \item $a_{ij}=a$ for some positive real number $a$ for every $j\leq i$, and $a$ is the minimal possible.
      
       \item Either (i): ``$X_{ii}\to X_{i0}$ is a chain and $a\leq 2$", or (ii) ``$a\leq 1$ and there is a blow-up of sections $Y_{ii}\to Y_{i(i-1)}\to \dots\to Y_{i0}\to X_{i1}$ of layer $i$ with blow-up centers $l_{ij}$ and $\operatorname{mult}_{l_{ij}}\tilde{B_i}=a$ and $c_{F_i}(Y_{ii})$ is a closed point contained in $l_{i(i-1)}$ for $0\leq j\leq i$".
       \item There are prime exceptional divisors $D_{i1}$, $D_{i2}$ on $X_{i1}$ over $X_{i0}$, a natural number $k$, and smooth curves $C_{i1},\dots,C_{ik}$ intersecting transversely mutually contained in $D_{i1}\cup D_{i2}$ such that $B_i\cap \Delta_i\subseteq C_{i1}\cup \dots\cup C_{ik
       }$ around $x_{i0}$.
       \item If $a\leq 1$, then $k=2$ and $\operatorname{mult}_{C_{i1}}B_i=\operatorname{mult}_{C_{i2}}B_i=a$.
       \item $a>1$.
       \item $\Delta_i$ is a prime divisor and $x_{i2}\not\in \tilde{\Delta_i}$.
       \item There is a non-negative integer $n'$, an $(n'+1)$-tuple of non-negative real numbers $\mathbf{a}'=(a_0'>1,a_1',\dots,a_n')$ and a positive real number $c'\leq 1$ such that $\mathrm{P}(I_0, m,\mathbf{a}',c')$ holds.
       \item Let $(\mathbb{C}^3, \tilde{S}_i)$ be pairs and $\Delta_i$ be prime divisors, $t$ and $m'$ be real numbers, $Z_{ii}\to Z_{i(i-1)}\to\dots\to Z_{i1}\to \mathbb{C}^3$ be birational models, $F_{ikl}$ be the exceptional divisors, and $x_{ij}$ be the blow-up centers given by $\mathrm{P}(I_0, m,\mathbf{a}',c')$ for $i=1,2, 3,\dots$ as in Definition \ref{P}. Let $\mathbf{M}=[-E_{\mathbf{o}}]$ be the b-divisor given as in Definition \ref{P}. Then, $i$ large enough, $\big(Z_{i1},\tilde{S}_i+\Delta_{i1}+(1-\sum\mathbf{a}')\mathbf{M}\big)/\mathbb{C}^3$ is not log canonical around $x_{i1}$ but $\big(Z_{i1},\frac{\sqrt{i}-1}{\sqrt{i}}\tilde{S}_i+\Delta_{i1}+(1-\sum\mathbf{a}')\mathbf{M}\big)/\mathbb{C}^3$ is log canonical around $x_{i1}$, where $\Delta_{i1}$ is the exceptional divisor over $\mathbb{C}^3$ such that $\tilde{S_i}+\Delta_{i1}+(1-\sum\mathbf{a}')\mathbf{M}$ is the crepant pullback of $S_i+(1-\sum\mathbf{a}')\mathbf{M}$ on $Z_{i1}$.
       \item Such sequences in Lemma \ref{I0} does not exist, which proves Theorem \ref{original}.
       
   \end{enumerate}
\end{thm}
\begin{proof}
\begin{enumerate}
    \item By Lemma \ref{lim}, we may assume either (1) holds or $\dim x_{i1}>0$ for all $i$. Cutting by a general hyperplane intersection, the later contradicts Theorem \ref{main2}.
    \item  Let $I^+_0=I_0\cap [0,\infty )$. Since $\operatorname{coeff}B_i\in I^+_0$, $\{ a_{ij}\}_{ij}\subseteq \sum I^+_0\cup \{ 0 \}$ is a discrete set. By the potential nefness of $\tilde{B}_i$ on $X_{i(j+1)}$ over $X_{i(j-1)}$, $a_{i(j+1)}\leq a_{ij}$ for all $i,j$. (2) then follows from Lemma \ref{lim}.
    \item Assume that (ii) does not hold. By Lemma \ref{lim}, we may assume that there is a natural number $h$ such that there is a tower of section-blow-ups $W_i\to X_i$ of layer $h$ with multiplicity $a$ along $B$ and $\operatorname{mult}_{{\bullet}}\tilde{B_i}< a$ around outside the closed point $c_{W_i}(F_i)$. Applying induction on $h$, (i) then follows from Lemma \ref{induction_on_layer}, Lemma \ref{strict_transform} and Lemma \ref{lim}.
    \item If (i) of (3) holds, this follows from Lemma \ref{finitecurve}. Otherwise, (ii) of (3) holds and we may assume that $x_{ih}\not \in \tilde{\Delta}_i$ for some fixed natural number $h\geq 3$ and $E_{ij}\cap \tilde{B_i}$ is a line for every $j<h$, on which the center of $F_i$ lies. (4) then follows from the log smoothness of $(X_{ih},\sum_{j=1}^k E_{ij})$.
    
    \item If (ii) of (3) holds, (5) follows from (4). Otherwise, the (i) of (3) holds and by Lemma \ref{finitecurve}, there is a curve $C_{i1}\subseteq \operatorname{Supp}\Delta_i$ with $\operatorname{mult}_{C_{i1}}B_i=a$. By Lemma \ref{curvebupinduction} and the boundedness of $\Delta_i$, we may assume that $C_{i1}$ is contained in only one prime component of $\Delta_i$. Lemma \ref{a<=1} then gives a contradiction that $F_i$ computes the minimal log discrepancy.
    \item This follows from Lemma \ref{a<=1}.
    \item By (6), (i) of (3) holds. This follows from Lemma \ref{finitecurve} and Lemma \ref{lim}.
    \item Since $X_{ii}\to X_{i0}$ is a chain, by (7) and Lemma \ref{a<=1}, $\mathrm{P}\big(I_0,m,(a),c\big)$ holds for some $c\leq a$. This follows from Lemma \ref{step8_2}.
    \item Since $(\mathbb{C}^3,\frac{\sqrt{i}-1}{\sqrt{i}}S_i+t\Delta_i)$ is sub-lc, a non-lc center $\eta_i$ of non-lc place $F'_i$ of $\big(Z_{i1},\frac{\sqrt{i}-1}{\sqrt{i}}\tilde{S}_i+\Delta_{i1}+(1-\sum\mathbf{a}')\mathbf{M}\big)/\mathbb{C}^3$ must be contained in $\operatorname{Supp}{\tilde{\Delta}_i}\cup\operatorname{Supp}\Delta_{i1}$. So, $\eta_i\subseteq F_{i1n}$ if $x_{i1}\in \eta_i$.
    
    
    Moreover, $\eta_i$ is a non-lc center of $\big(Z_{ik},\tilde{S_i}+\Delta_{ik}+(1-\sum\mathbf{a})\mathbf{M}\big)/\mathbb{C}^3$ for all $1\leq k\leq i$. By Lemma \ref{a<=1} and the sixth condition of Definition \ref{P}, $c_{Z_{ii}}(F'_i)$ is a closed point in $F_{iin}$ outside $\tilde{F_{ikl}}$ for any $(k,l)\neq (i,n)$. Let $D_i=\frac{1}{\sqrt{i}}\tilde{S}_i\subseteq Z_{i1}$ be a divisor on $Z_{i1}$. Since $\operatorname{ord}_{F_{iin} }(t\Delta_i)+\operatorname{ord}_{F_{iin}}(D_i)=t+\frac{i-1}{\sqrt{i}}(\sum\mathbf{a}'+n+2)>1-\sum \mathbf{a}'=(1-\sum\mathbf{a}')\operatorname{ord}_{F_{iin}}(E_{\mathbf{o}})$ for $i$ large enough, $a_{F'_{i}}\big(Z_{i1},\frac{\sqrt{i}-1}{\sqrt{i}}\tilde{S}_i+\Delta_{i1}+(1-\sum\mathbf{a}')\mathbf{M}\big)>a_{F'_{i}}(\mathbb{C}^3,S_i+t\Delta_i)=m\geq 0$, a contradiction. 

    On the other hand, suppose that $F_i'$ is a mld place of $(\mathbb{C}^3,S_i+t\Delta_i)$ centered at $\mathbf{o}$. Then $a_{F'_i}\big(Z_{i1},\tilde{S}_i+\Delta_{i1}+(1-\sum\mathbf{a}')\mathbf{M}\big)=a_{F'_i}\big(\mathbb{C}^3,S_i+(1-\sum\mathbf{a}')\mathbf{M}\big)>a_{F_{i1n}}\big(\mathbb{C}^3,S_i+(1-\sum\mathbf{a}')\mathbf{M}\big)=0$, which proves (7).
    \item By (9), $\operatorname{lct}\big(Z_{i1}/\mathbb{C}^3,\Delta_{i1}+(1-\sum\mathbf{a}')\mathbf{M};\tilde{S}_i\big)\in [\frac{\sqrt{i}-1}{\sqrt{i}},1)$ for all large $i$, which contradicts the local acc for lct (c.f. \cite{HX26}).
    
\end{enumerate}
\end{proof}


\bibliographystyle{alpha-links}
\bibliography{ref}

\end{document}